%% file: arXiv.tex
\documentclass[11pt]{article}

\usepackage{biblatex}

\usepackage{microtype}
\usepackage{subfigure}
\usepackage{amsmath, amssymb, amsthm}
\usepackage{bm}
\usepackage{amsfonts}
\usepackage{xcolor}
\usepackage{hyperref}
\usepackage{graphicx}

\usepackage{fullpage}
\usepackage{hyperref}
\usepackage{cleveref}
\newtheorem{theorem}{Theorem}
\newtheorem{othertheorem}{othertheorem}[section]
\newtheorem{lemma}[othertheorem]{Lemma}
\newtheorem{corollary}[othertheorem]{Corollary}
\newtheorem{proposition}[othertheorem]{Proposition}

\newtheorem{definition}[othertheorem]{Definition}

\newtheorem{rem}[othertheorem]{Remark}

\newcommand{\norm}[1]{\lVert#1\rVert}

\numberwithin{equation}{section}

\date{}

\usepackage{algorithmic}
\usepackage[ruled,vlined]{algorithm2e}
\graphicspath{{Image/}}
\usepackage{authblk}
\begin{document}

\title{Asymptotic Statistical Analysis of Sparse Group LASSO via Approximate Message Passing Algorithm}

 \author[]{Kan Chen\thanks{Email: {\tt kanchen@sas.upenn.edu}} }
 \author[]{Zhiqi Bu \thanks{Email: {\tt zbu@sas.upenn.edu}} }
 \author[]{Shiyun Xu \thanks{Email: {\tt shiyunxu@sas.upenn.edu}}}

\affil[]{Graduate Group of Applied Mathematics and Computational Science
\\
University of Pennsylvania}
\maketitle

\begin{abstract}
Sparse Group LASSO (SGL) is a regularized model for high-dimensional linear regression problems with grouped covariates. SGL applies $l_1$ and $l_2$ penalties on the individual predictors and group predictors, respectively, to guarantee sparse effects both on the inter-group and within-group levels. In this paper, we apply the approximate message passing (AMP) algorithm to efficiently solve the SGL problem under Gaussian random designs. We further use the recently developed state evolution analysis of AMP to derive an asymptotically exact characterization of SGL solution. This allows us to conduct multiple fine-grained statistical analyses of SGL, through which we investigate the effects of the group information and $\gamma$ (proportion of $\ell_1$ penalty). With the lens of various performance measures, we show that SGL with small $\gamma$ benefits significantly from the group information and can outperform other SGL (including LASSO) or regularized models which does not exploit the group information, in terms of the recovery rate of signal, false discovery rate and mean squared error.
\end{abstract}

\input{File/001Introduction.tex}

\input{File/002AMP_A.tex}

\input{File/003Main_Result.tex}

\input{File/004Simulation.tex}

\input{File/005Conclusion_Future_Work.tex}

\printbibliography
\clearpage
\appendix

\input{File/006Appendix.tex}

\end{document}

%% file: File/001Introduction.tex
\section{Introduction}
\label{Introduction}

Suppose we observe an $n \times p$ design matrix $\mathbf{X}$, and the response $\mathbf{y}\in\mathbb{R}^n$ which is modeled by 
\begin{equation} \label{linear model}
    \mathbf{y} = \mathbf{X} \bm \beta + \bm w
\end{equation}
in which $\bm w\in\mathbb{R}^n$ is a noise vector. \smallskip

In many real life applications, we encounter the $p \gg n$ case in which standard linear regression fails to have a unique solution. To address this issue, \cite{tibshirani1996regression} introduced the LASSO by adding the $\ell_1$ penalty, whose magnitude is controlled by a scalar $\lambda>0$, and minimizing
\begin{equation}
    \frac{1}{2} \|\mathbf{y} - \mathbf{X} \bm \beta\|_2^2 + \lambda \|\bm \beta\|_1.
\end{equation}

The LASSO is arguably the most popular method that allows to select features or predictors, in the sense that the minimizer to the LASSO problem can be sparse, i.e. some entries of the minimizer are zero and hence not viewed as significant predictors. 

Suppose further that the predictors are divided into multiple groups. One could select a few of the groups rather than a few components of $\bm \beta$. An analogous procedure for group selection is Group LASSO \cite{yuan2006model}:
\begin{equation}
    \min_{\bm \beta} \frac{1}{2} \|\mathbf{y} - \sum_{l=1}^L \mathbf{X}_{l}\bm \beta_{l} \|_2^2 + \lambda \sum_{l=1}^L \sqrt{p_l} \|\bm \beta_l\|_2
\end{equation}
where $\mathbf{X}_l$ represents the predictors corresponding to the $l$-th group, with corresponding coefficient vector $\bm\beta_l$. Here the group information $\bm g\in\mathbb{R}^p$ is implicit in the definition of $\mathbf{X}_l$, indicating that the $i$-th predictor belongs to the $g_i$-th group. We assume that $p$ predictors are divided into $L$ groups and denote the size of the $l$-th group as $p_l$.

However, for a group selected by Group LASSO, all entries in the group are nonzero simultaneously. To yield both sparsity of groups (like Group LASSO) and sparsity within each group (like the LASSO), \cite{simon2013sparse} introduced the Sparse Group LASSO problem as follows:
\begin{equation} \label{objective}
    \min_{\beta \in \mathbb{R}^p}\frac{1}{2}\|\mathbf{y} - \sum_{l = 1}^{L} \mathbf{X}_l \bm\beta_l \|_2^2  + (1 - \gamma)\lambda \sum_{l=1}^L \sqrt{p_l} \|\bm\beta_l\|_2 +  \gamma \lambda \|\bm\beta\|_1
\end{equation}
where $\gamma \in [0,1]$ refers to the proportion of the LASSO fit in the overall penalty. If $\gamma = 1$, SGL is purely LASSO, while if $\gamma=0$, SGL reduces to Group LASSO. We denote the solution to the SGL problem as $\bm{\hat\beta}$. \smallskip

As a convex optimization problem, SGL can be solved by many existing methods, including ISTA (Iterative Shrinkage-Thresholding Algorithm) \cite{daubechies2004iterative} and FISTA (Fast Iterative Shrinkage-Thresholding Algorithm) \cite{beck2009fast}. Both methods rely on the derivation of the proximal operator of the SGL problem, which is non-trivial due to the \textit{non-separability} of the penalty when $\gamma\in[0,1)$. In previous literature, \cite{simon2013sparse} used the blockwise gradient descent algorithm with backtracking, instead of the proximal approach, to solve this problem. Other method like fast block gradient descent \cite{ida2019fast} can also be implemented and will be compared in \Cref{fig:MSE1}. 

In this paper, we establish the \textit{approximate message passing} (AMP) algorithm \cite{bayati2011lasso,donoho2011design,bayati2011dynamics,bayati2013estimating,donoho2009message,montanari2012graphical}  for SGL on top of its proximal operator. We then analyze the algorithmic aspects of SGL via the AMP theory. In general, AMP is a class of computationally efficient gradient-based algorithms originating from graphical models and extensively studied for many compressed sensing problems \cite{krzakala2012probabilistic,rangan2011generalized}.

To be specific, we derive, for fixed $\gamma$, the SGL AMP as follows: set $\bm\beta^0=\bm 0, \bm z^0=y$ and for $t>0$,
\begin{align}
    \bm \beta^{t+1} &= \eta_{\gamma}(\mathbf{X}^\top \mathbf{z}^t + \bm \beta^t, \theta_t)
\\
\mathbf{z}^{t+1} &= \mathbf{y} - \mathbf{X} \bm \beta^{t+1} + \frac{1}{\delta} \mathbf{z}^{t} \langle \eta_{\gamma}'(\mathbf{X}^\top \mathbf{z}^t + \bm \beta^t, \theta_t) \rangle.
\end{align}
Here the threshold $\theta_t$ is carefully designed and will be discussed in details in \Cref{Main Result}. Note that $\langle\bm v\rangle:=\sum_{i=1}^p v_i/p$ is the average of vector $\bm v$. Furthermore, $\eta_\gamma$ is the \textit{proximal operator}
\begin{align*}
\eta_\gamma(\bm s,\lambda)&:=\underset{\bm b}{\text{argmin}}\frac{1}{2}\|\bm s-\bm b\|^2
+(1 - \gamma)\lambda \sum_{l=1}^L \sqrt{p_l} \|\bm b_l\|_2 +  \gamma \lambda \|\bm b\|_1
\end{align*}
and $\eta_\gamma':=\nabla\circ\eta_\gamma$ is the diagonal of the Jacobian matrix of the proximal operator with respect to its first argument $\bm s$, with $\circ$ being the Hadamard product.

\textit{Empirically}, the simulation results in Figure \ref{fig:MSE1} and Table \ref{tab:MSE1} demonstrate the supremacy of AMP convergence speed over the two most well-known proximal gradient descent methods, ISTA and FISTA. We also compare these methods to the Nesterov-accelerated blockwise descent in \cite{simon2013sparse} and in R package \texttt{SGL}. We note that the Nesterov-accelerated ISTA (i.e. FISTA) outperforms the accelerated blockwise descent in terms of both the number of iterations and the wall-clock time (see Figure \ref{fig:times}(a) and the detailed analysis in Appendix \ref{app:proximal}). This observation suggests that using the proximal operator not only requires fewer iterations but also reduces the complexity of computation at each iteration. We pause to emphasize that, in general, the cost function $\mathcal{C}_{\bm X,\bm y}(\bm\beta):=\frac{1}{2}\|\mathbf{y} - \sum_{l = 1}^{L} \mathbf{X}_l \bm\beta_l \|_2^2  + (1 - \gamma)\lambda \sum_{l=1}^L \sqrt{p_l} \|\bm\beta_l\|_2 +  \gamma \lambda \|\bm\beta\|_1$ is not strictly convex. Therefore, we choose the optimization error (mean squared error, or MSE, between $\bm\beta^t$ and $\bm\beta$) as the measure of convergence, as there may exist $\hat{\bm\beta}$ far from $\bm\beta$ for which $\mathcal{C}(\hat{\bm\beta})$ is close to $\mathcal{C}(\bm\beta)$.

\begin{table}[!htb]
    \centering
    \begin{tabular}{ |p{1cm}||p{1cm}|p{1cm}|p{1cm}|p{1cm}|  }
 \hline
 \multicolumn{5}{|c|}{Number of iterations to reach certain precision} \\
 \hline 
 \textbf{MSE}   & $10^{-2}$    &$10^{-3}$&  $10^{-4}$ & $10^{-5}$ \\
 \hline
 ISTA   & 309    &629&   988 & 1367 \\
 FISTA    &   42  & 81   &158 & 230 \\
 AMP &4 & 6&  14 & 35 \\
 \hline
\end{tabular}
    \caption{Same settings as in Figure \ref{fig:MSE1} except $\lambda =1$ and the prior $\mathbf{\bm\beta}_0 $ is $5\times$Bernoulli(0.1).}
    \label{tab:MSE1}
\end{table}

In \cite{bayati2011lasso,mousavi2018consistent}, it has been rigorously proved that applying AMP to LASSO shows nice statistical properties, such as the exact characterization of the LASSO solution. However, applying AMP to a non-separable regularizer is more challenging. Along this line of research, recently, SLOPE AMP was rigorously developed in \cite{bu2019algorithmic} and \cite{celentano2019approximate} further analyzed a class of non-separable penalties which become asymptotically separable. Nevertheless, the question of whether AMP provably solves and characterizes SGL remains open for researchers.

\begin{figure}[!htb]
\centering
\includegraphics[width=0.5\linewidth]{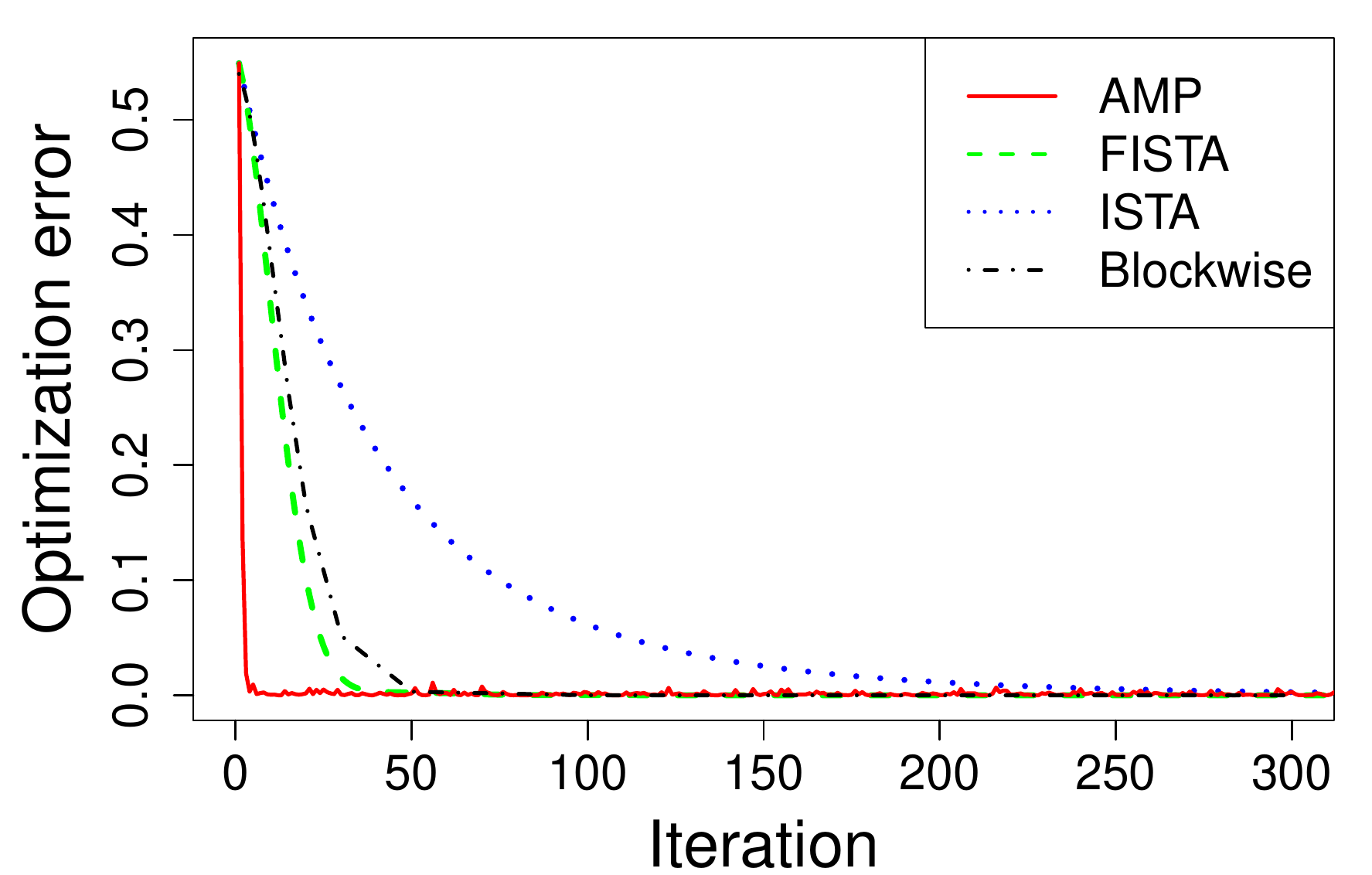}
\vspace{-0.2cm}
\caption{$p = 4000, n = 2000, \gamma = 0.5, \bm g=(1,\cdots,1)$, the entries of $\mathbf{X}$ are i.i.d. $\mathcal{N}(0,1/n)$, $\lambda = 2$, and the prior $\mathbf{\beta} $ is $5\times$Bernoulli(0.2).}
\label{fig:MSE1}
\end{figure}

Our contributions are as follows. We first derive a proximal operator of SGL on which the SGL AMP algorithm is based. We prove in \Cref{theorem2} that the algorithm solves the SGL problem \textit{asymptotically exactly} under i.i.d. Gaussian designs. The proof leverages the recent state evolution analysis \cite{berthier2017state} for non-separable penalties and shows that the state evolution characterizes the asymptotically exact behaviors of $\hat{\bm\beta}$. Specifically, the distribution of SGL solution is completely specified by a few parameters that are the solution to a certain fixed-point equation asymptotically. As a consequence, we can use the characterization of the SGL solution to analyze the behaviors of the $\hat{\bm\beta}$ precisely, as shown in \Cref{theorem1}. In particular, we investigate the effects of group information and $\gamma$ on the model performance. We empirically find that SGL can benefit significantly from good group information in the sense of MSE, true positive proportion (TPP), and false discovery proportion (FDP), while suffering from bad group information.

The rest of this paper is divided into four sections. In Section 2, we give some preliminary background of the AMP algorithm. In Section 3, we state our main theorems about the convergence and the characterization. In Section 4, we show the simulation results to support our theoretical results. Lastly, in Section 5, we conclude our paper and list some possible extensions of future work.

%% file: File/002AMP_A.tex
\section{Algorithm}
\label{Preliminary}

\subsection{Approximate Message Passing}
In order to guarantee that the AMP algorithm works for SGL, we need to make several assumptions.
\paragraph{Assumptions for AMP}
\begin{itemize} \label{assumptions}
    \item (\textbf{A1}) The measurement matrix $\mathbf{X}$ has independent entries following $\mathcal{N}(0, \frac{1}{n})$.
    \item (\textbf{A2}) The elements of signal $\bm \beta$ are i.i.d. copies of a random variable $\Pi$ with $\mathbb{E}(\mathrm{\Pi}^2 \max \{0, \log(\mathrm{\Pi})  \} ) < \infty$. We use $\mathbf{\Pi} \in \mathbb{R}^p$ to denote random vector with each component following i.i.d. $\mathrm{\Pi}$. 
    \item (\textbf{A3}) The elements of noise $\bm w$ are i.i.d. $\mathrm{\Pi}$ with $\sigma_{\bm w}^2:=\mathbb{E}(W^2 ) < \infty$.
    \item (\textbf{A4}) The ratio of sample size to feature size $\frac{n}{p}\to\delta \in (0, \infty)$ as $n,p \rightarrow \infty$.
\end{itemize}
We note that the assumptions are the same as those for the LASSO (and its generalization) in \cite{bu2019algorithmic}, and the second-moment assumptions \textbf{(A2)} and \textbf{(A3)} can be relaxed. For example, we can instead assume that $\bm w$ has an empirical distribution that converges weakly to $W$, with $\|\bm w\|^2/p\to\mathbb{E}(W^2)<\infty$. In general, we may extend assumptions \textbf{(A1)} and \textbf{(A2)} to a much broader range of cases, as discussed in Section \ref{extensions}.
Additionally, we need one extra assumption specifically designed for the group information in SGL as follows.
\begin{itemize}
    \item (\textbf{A5}) The relative ratio of each group size, $p_l/p\to r_l \in (0, 1)$ as $p \rightarrow \infty$.
\end{itemize}
Now we write the SGL AMP algorithm based on \cite{donoho2011design}:
\begin{equation}\label{recursion 5}
    \bm{\beta}^{t+1} = \eta_{\gamma,\bm g}(\mathbf{X}^\top \mathbf{z}^t + \bm{\beta}^t, \alpha \tau_t)
\end{equation}
\begin{equation}\label{recursion 6}
    \mathbf{z}^{t+1} = \mathbf{y} - \mathbf{X} \bm{\beta}^{t+1} + \frac{1}{\delta} \mathbf{z}^{t} \langle \eta_{\gamma,\bm g}'(\mathbf{X}^\top \mathbf{z}^t + \bm{\beta}^t, \alpha \tau_t) \rangle
\end{equation}
\begin{equation}\label{state evolution}
    \tau_{t+1}^2  = \sigma_{\bm w}^2 + \lim_{p \rightarrow \infty}\frac{1}{\delta p} \mathbb{E}\|\eta_{\gamma,\bm g}(\mathbf{\Pi} + \tau_t \bm Z, \alpha \tau_t) - \mathbf{\Pi}\|_2^2
\end{equation}
where $\bm Z$ is the standard Gaussian $\mathcal{N}(0,\mathcal{I}_p)$ and the expectation is taken with respect to both $\mathbf{\Pi}$ and $\bm Z$. We denote $\eta_{\gamma,\bm g}(\bm s,\lambda): \mathbb{R}^{p}\times\mathbb{R} \rightarrow \mathbb{R}^p$ as the proximal operator for SGL, which we will derive in Section \ref{sec:proximal}. We notice that, comparing AMP to the standard proximal gradient descent, the thresholds are related to $(\alpha,\tau_t)$ instead of to $\lambda$. On one hand, $\tau_t$ is derived from equation \eqref{state evolution}, known as the \textbf{state evolution}, which relies on $\alpha$. On the other hand, $\alpha$ corresponds uniquely to $\lambda$ via equation (\ref{calibration}) which is so called \textbf{calibration}: with $\tau_*:=\lim_{t\to\infty}\tau_t$,
\begin{equation}\label{calibration}
    \lambda = \alpha \tau_*\left(1 - \lim_{p \rightarrow \infty}\frac{1}{\delta} \langle \eta_{\gamma,\bm g}'(\mathbf{\Pi}+ \tau_*\bm Z, \alpha \tau_*) \rangle \right).
\end{equation}

\subsection{Proximal Operator and Derivative }
\label{sec:proximal}
In this section we derive the proximal operator  \cite{daubechies2004iterative} for SGL. In comparison to \cite{simon2013sparse,foygel2010exact,friedman2010note}, which all used subgradient conditions to solve SGL, our proximal-based methods can be much more efficient in terms of convergence speed and accuracy (for details of comparison, see Appendix \ref{app:proximal}).

Denote $ \beta^{(j)}$ as the $j$-th component of $\bm\beta$ and the cost function $\mathcal{C}_{\lambda, \gamma,\bm g}$ as
\begin{equation*} \label{cost}
   \begin{split}
        &\mathcal{C}_{\lambda, \gamma,\bm g}(s,\bm \beta)
        = \frac{1}{2}\|s -  \bm \beta \|_2^2 + (1 - \gamma)\lambda \sum_{l=1}^L \sqrt{p_l} \|\bm \beta_l\|_2 +  \gamma \lambda \|\bm \beta\|_1.
   \end{split}
\end{equation*}
We sometimes ignore the dependence on the subscripts when there is no confusion. When $g_j=l$ and $\|\bm \beta_l\|_2\neq 0$, we set $\frac{\partial\mathcal{C}}{\partial \beta^{(j)}}=0$ and denote $l_j:=\{i: g_i=g_j\}$. Then the explicit formula is
\begin{equation} \label{eta}
    \eta_{\gamma}(\bm s, \lambda)^{(j)} =  \eta_{\text{soft}}(s^{(j)}, \gamma  \lambda)\left(1 - \frac{(1 - \gamma)\lambda \sqrt{p_{l_j}}}{ \| \eta_{\text{soft}}(\bm s_{l_j}, \gamma \lambda) \|_2 }\right)
\end{equation}
when $\eta_{\text{soft}}(s_{l_j}, \gamma \lambda)\in\mathbb{R}^{p_{l_j}}$ has a non-zero norm. Here 
\begin{equation}
    \eta_{\text{soft}}(x; b) = \begin{cases} 
      x - b & x > b \\
       x + b & x < -b  \\
      0 & \text{otherwise}
   \end{cases}
\end{equation}
is the soft-thresholding operator. 

We emphasize that \eqref{eta} is incomplete due the the non-differentiability of $\|\bm\beta_l\|_2$ at $\bm\beta_l=0$. Denoting equation \eqref{eta} as $\hat\eta_\gamma^{(j)}$, we have the full formula of the proximal operator as
\begin{align}
\eta_\gamma(\bm s,\lambda)^{(j)}=
\begin{cases}
\hat\eta_\gamma^{(j)} & \text{ if $\|\eta_{\text{soft}}(\bm s_{l_j},\gamma\lambda)\|_2 > (1-\gamma)\lambda \sqrt{p_{l_j}}$}
\\
0 & \text{ otherwise }
\end{cases}
\label{eq:prox_op}
\end{align}
The details of derivation are left to Appendix \ref{app:operator}.

Similarly, the derivative of the proximal operator w.r.t. $s$ also has two forms, split by the same conditions as above. For simplicity we only describe the non-zero form here: 
\begin{equation}
        \eta_\gamma'(\bm s,  \lambda)^{(j)} = \textbf{1}\{ |s^{(j)}| > \gamma \lambda \}
        \cdot \left[1  - \frac{(1 - \gamma) \lambda \sqrt{p_{l_j}}}{\|\eta_{\text{soft}}(\bm s_{l_j},\gamma\lambda)\|_2} \left(1 - \frac{\eta_{\text{soft}}^2(s^{(j)}, \gamma \lambda)}{\|\eta_{\text{soft}}(\bm s_{l_j},\gamma\lambda)\|_2^2}\right) \right]
\end{equation}

We note that in \cite{simon2013sparse}, the SGL problem is divided into $L$ sub-problems according to the group information. Then subgradient conditions are used to construct majorization-minimization problems. These problems are solved cyclically, via accelerated gradient descent, in a blockwise or groupwise manner. Along this line of research, there have been other blockwise descent methods designed for Group LASSO \cite{yang2015fast}. In contrast, our proximal operator is unified, as can be seen by comparing Algorithm \ref{alg:block} and Algorithm \ref{alg:ista} in Appendix \ref{app:proximal}. We note that our proximal operator also has a groupwise flavor, but all groups are updated independently and simultaneously, thus improving the convergence speed. Since SGL AMP is built on the proximal operator, we will validate the proximal approach by introducing ISTA and FISTA for SGL, with detailed complexity analysis and convergence analysis in Appendix \ref{app:proximal}.

%% file: File/003Main_Result.tex
\section{Main Results}
\label{Main Result}

\subsection{State Evolution and Calibration}
Notice that in SGL AMP, we use $\theta_t$ as the threshold, whose design requires state evolution and calibration. Thus we study the properties of the state evolution recursion \eqref{state evolution}. To simplify the analysis, we consider the finite approximation of state evolution and present precise conditions which guarantee that the state evolution converges efficiently.

\begin{proposition}
\label{prop:F}
Let $\mathbf{F}_\gamma (\tau_t^2, \alpha \tau_t)= \sigma_{\bm w}^2 + \frac{1}{\delta p} \mathbb{E}\|\eta_{\gamma}(\mathbf{\Pi} + \tau_t \bm Z, \alpha \tau_t) - \mathbf{\Pi} \|_2^2$
and define $\mathcal{A}(\gamma) = \{ \alpha :\delta \geq 2T(\gamma \alpha) - 2(1- \gamma)\alpha \sqrt{2T(\gamma \alpha)} + (1 - \gamma)^2 \alpha^2 \}$
with $T(z) = (1 + z^2) \Phi(-z) - z \phi(z)$, $\phi(z)$ being the standard Gaussian density and $\Phi(z) = \int_{-\infty}^z \phi(x) dx$. For any $\sigma_{\bm w}^2 > 0$, $\alpha \in \mathcal{A}(\gamma) $ , the fixed point equation $\tau^2 = \mathbf{F}_{\gamma}(\tau^2, \alpha \tau)$ admits a unique solution. Denoting the solution as $\tau_* = \tau_* (\alpha)$, we have $\lim_{t \rightarrow \infty}\tau_t \rightarrow \tau_*(\alpha)$, where the convergence is monotone under any initial condition. Finally $\left|\frac{d\mathbf{F}_{\gamma}}{d\tau^2}\right| < 1$ at $\tau = \tau_*$
\end{proposition}

A demonstration of $\mathcal{A}$ is given in Figure \ref{alpha}. We note that for all $\gamma<1$, $\mathcal{A}$ has upper and lower bounds; however, when $\gamma=1$, i.e. for LASSO, there is no upper bound. We provide the proof of \Cref{prop:F} in Appendix \ref{app:se}. 

\vspace{-0.5cm}
\begin{figure}[!htb]
\centering
    \includegraphics[width=5.5cm]{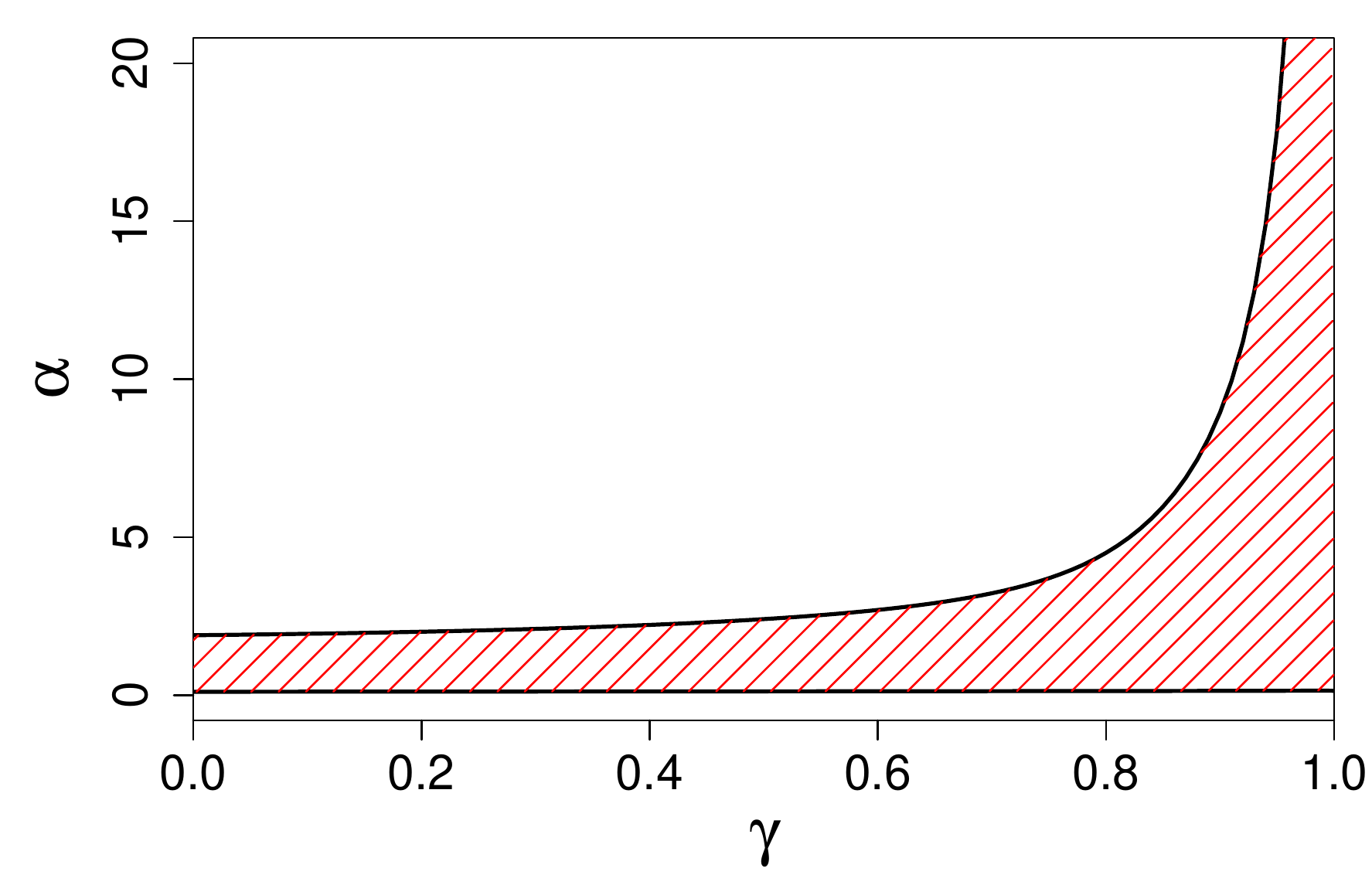}
    \vspace{-0.5cm}
\caption{$\mathcal{A}$ when $\delta=0.2$ is represented in the red shaded region.}
    \label{alpha}
\end{figure}
\vspace{-0.5cm}
\begin{rem}
When $\alpha$ is outside the set $\mathcal{A}$ in Proposition \ref{prop:F}, we must have $\alpha>\mathcal{A}_{\max}$, since we use a non-negative $\lambda$ as penalty which guarantees $\alpha > \mathcal{A}_{\min}$. To see this, consider $\alpha\not\in\mathcal{A}$, we have $\tau=\infty$ and hence the dominant term in $\pi+\tau\bm{Z}$ is $\tau\bm{Z}$. We can view $\pi$ as if vanishing and easily derive $\alpha>\mathcal{A}_{\min}$ from the state evolution. For $\alpha > \mathcal{A}_{\max}$, the state evolution in fact still converges. \footnote{We note that $\mathcal{A}$ is a sufficient but not necessary condition for the state evolution to converge. The reason that we split the analysis at $\alpha=\mathcal{A}_{\max}$ is because, for $\alpha > \mathcal{A}_{\max}$, the SGL estimator is 0. Besides, we note that the set $\mathcal{A}$ only affects the state evolution. Hence when $\alpha > \mathcal{A}_{\max}$, the calibration is still valid and the mapping between $\alpha$ and $\lambda$ is monotone.}
\end{rem}

Before we employ the finite approximation of state evolution to describe the calibration (\ref{calibration}), we explain the necessity of calibration by the following lemma. 

\begin{lemma}
\label{lemma:3.2}
For fixed $\gamma$, a stationary point $\hat{\bm{\beta}}$ with corresponding $\hat{\bm{z}}$ of the AMP iteration (\ref{recursion 5}), (\ref{recursion 6}) with $\theta_t = \theta_*$ is a minimizer of the SGL cost function in \eqref{objective} with
$ \lambda = \theta_* \left(1 - \frac{1}{\delta} \langle \eta_\gamma'(\mathbf{X}^\top \hat{\mathbf{z}} + \hat{\bm{\beta}}, \theta_*) \rangle   \right)$.
\end{lemma}

Setting $\theta_* = \alpha \tau_*$, we are in the position to define the finite approximation of calibration between $\alpha$ and $\lambda$ by
\begin{equation}
    \lambda = \alpha \tau_{*} \left(1 - \frac{1}{\delta}  \langle\eta_{\gamma}'(\mathbf{\Pi} + \tau_* \bm Z, \alpha \tau_*)\rangle  \right)
    \label{eq:cali_finite}
\end{equation}

In practice, we need to invert (\ref{eq:cali_finite}) to input $\lambda$ and recover
\begin{equation}
    \alpha (\lambda) \in \{ a  \in \mathcal{A}: \lambda(a) = \lambda  \}
\end{equation}
The next proposition and corollary imply that the mapping of $\lambda \rightarrow \alpha (\lambda)$ is well-defined and easy to compute.

\begin{proposition} 
\label{prop:bijective}
  The function $\alpha \rightarrow \lambda(\alpha)$ is continuous on $\mathcal{A}(\gamma)$ with $\lambda(\min\mathcal{A}) = -\infty$ and $\lambda(\max\mathcal{A}) = \lambda_{\max}$ for some constant $\lambda_{\max}$ depending on $\Pi$ and $\gamma$. Therefore, the function $\lambda \rightarrow \alpha(\lambda)$ satisfying $\alpha(\lambda) \in \{ \alpha \in \mathcal{A}(\gamma): \lambda(\alpha) = \lambda \}$ exists where $\lambda \in (-\infty, \lambda_{\max})$.
\end{proposition}
Given $\lambda$, Proposition \ref{prop:bijective} claims that $\alpha$ exists and the following result guarantees its uniqueness.
\begin{corollary}
\label{cor:monotone}
For $\lambda<\lambda_{\max}, \sigma_{\bm w}^2 >0$, there exists a unique $\alpha \in \mathcal{A}(\gamma)$ such that $\lambda(\alpha) = \lambda$ as defined in \eqref{eq:cali_finite}. Hence the function $\lambda \rightarrow \alpha(\lambda)$ is continuous and non-decreasing with $\alpha((-\infty, \lambda_{\max})) = \mathcal{A}(\gamma)$.
\end{corollary}

The proofs of these statements are left in Appendix \ref{app:cali}.

\subsection{AMP Characterizes SGL Estimate Under Pseudo-Lipschitz Functions}

Having described the state evolution, we now state our main theoretical results. We establish an asymptotic equality between $\bm {\hat\beta}$ and $\eta_\gamma$ in pseudo-Lipschitz norm, which allows the fine-grained statistical analysis of the SGL minimizer.

\begin{definition}[c.f. \cite{berthier2017state}]
	\label{lipschitz}
	For $k\in \mathbb{N}_{+}$, a function $\phi : \mathbb{R}^{d} \to \mathbb{R}$ is \textbf{pseudo-Lipschitz} of order $k$, if there exists a constant $L$ such that for $\mathbf{a}, \mathbf{b} \in \mathbb{R}^{d}$,
	\begin{equation}
	\begin{split}
	   & |\phi(\bm a)-\phi(\bm b)| \leq L\Big(1+ \left(\frac{\|\bm a\|}{\sqrt{d}}\right)^{k-1}+ \left(\frac{\|\bm b\|}{\sqrt{d}}\right)^{k-1}\Big) \left(\frac{\|\bm a - \bm b\|}{\sqrt{d}}\right).
	\end{split}
	\end{equation}
	A sequence (in $p$) of pseudo-Lipschitz functions $\{\phi_{p}\}_{p\in \mathbb{N}_{+}}$
	is \textbf{uniformly pseudo-Lipschitz}
	of order $k$ if, denoting by $L_{p}$ the pseudo-Lipschitz constant of $\phi_{p}$, $L_{p} < \infty$ for each $p$ and $\lim\sup_{p\to\infty}L_{p} < \infty$.
\end{definition}
\begin{theorem}\label{theorem1}
 Under the assumptions (\textbf{A1})-(\textbf{A5}), for any uniformly pseudo-Lipschitz sequence of functions $\varphi_p :  \mathbb{R}^p \times \mathbb{R}^p \rightarrow \mathbb{R}$ and for $\mathbf{Z} \sim \mathcal{N}(0, \mathcal{I}_p), \mathbf{\Pi} \sim p_{\Pi}$,
 \begin{align*}
    \lim_{p \rightarrow \infty}
 \varphi_p(\bm{\hat\beta}, \bm \beta) =  \lim_{t\to\infty} \lim_{p\to\infty} \mathbb{E}[\varphi_p (\eta_{\gamma}(\mathbf{\Pi} + \tau_t \bm Z; \alpha \tau_t), \mathbf{\Pi} )].
 \end{align*}
\end{theorem}

The proof of \Cref{theorem1} is left in Appendix \ref{Appendix C}. Essentially, up to a uniformly pseudo-Lipschitz loss, we can replace $\bm{\hat\beta}$ by $\eta_\gamma$ in the large system limit. The distribution of $\eta_\gamma$ is \textit{explicit}, thus allowing the analysis of certain statistical quantities in an exact manner. Specifically, if we use $\varphi_p(\bm a,\bm b) = \frac{1}{p}\|\bm a - \bm b\|_2^2$, the MSE of estimation between $\bm{\hat\beta}$ and $\bm\beta$ can be characterized by $\tau_*$.
\begin{corollary} 
\label{cor:mse}
Under the assumptions (\textbf{A1})-(\textbf{A5}), then almost surely
\begin{equation*}
    \lim_{p \rightarrow \infty} \frac{1}{p} \|  \bm{\hat\beta}-\bm \beta\|_2^2 = \delta(\tau_*^2-\sigma_{\bm w}^2)
\end{equation*}
\end{corollary}

\begin{proof}[Proof of Corollary \ref{cor:mse}]
Applying Theorem \ref{theorem1} to the pseudo-Lipschitz loss function and letting $\varphi_p (\bm a,\bm b) = \|\bm a - \bm b \|_2^2$, we obtain
\begin{equation*}
    \lim_{p \rightarrow \infty} \|  \hat{\bm \beta} - \bm \beta \|_2^2= \lim_{t \rightarrow \infty} \mathbb{E}\left[\varphi_p (\eta_{\gamma}(\mathbf{\Pi} + \tau_t \bm Z; \alpha \tau_t), \mathbf{\Pi} )\right].
\end{equation*}
The result follows from the state evolution \eqref{state evolution} since 
\begin{align*}
    \lim_{t \rightarrow \infty} \mathbb{E}\left[\varphi_p (\eta_{\gamma}(\mathrm{\Pi} + \tau_t Z; \alpha \tau_t), \mathrm{\Pi} )   \right] = \delta(\tau_*^2-\sigma_{\bm w}^2).
\end{align*}
\end{proof}

Now that we have demonstrated the usefulness of our main theoretical result, we prove Theorem \ref{theorem1} at a high level, following the road map below where the equivalence is in the limiting sense.
\begin{center}
\fbox{$\bm{\hat\beta}\overset{\text{Theorem 2}}{=\joinrel=\joinrel=}\bm\beta^t\overset{\text{Lemma 3.8}}{=\joinrel=\joinrel=}\eta_{\gamma}(\mathbf{\Pi} + \tau_t \bm Z;\alpha\tau_t)$}
\end{center}

We first show the convergence of $\bm\beta^t$ to $\bm {\hat\beta}$, i.e. the AMP iterates converge to the true minimizer.


\begin{theorem}\label{theorem2} 
Under assumptions (\textbf{A1})-(\textbf{A5}), for the output of the AMP algorithm in (\ref{recursion 5}) and the Sparse Group LASSO estimator given by the solution of (\ref{objective}),
  \vspace{-0.3cm}
 \begin{equation*}
   \lim_{p \rightarrow \infty} \frac{1}{p} 
 \|\bm {\hat \beta} - \bm \beta^t \|_2^2 = k_t, \quad where \lim_{t \rightarrow \infty} k_t = 0
 \vspace{-0.2cm}
 \end{equation*}
\end{theorem}

The proof is similar to the proof of \cite[Theorem 1.8]{bayati2011lasso}. The difference is incurred by the existence of the $\ell_2$ norm which imposes the group structure. We leave the proof in Appendix \ref{Appendix C}. In addition to Theorem \ref{theorem2}, we borrow the state evolution analysis from \cite{berthier2017state} to complete the proof of Theorem \ref{theorem1}, whose proof can be found in \Cref{proof 38} as well.

\begin{lemma}[\cite{berthier2017state},Theorem 14]
\label{berthier}
Under assumptions (\textbf{A1}) - (\textbf{A5}), given that \textbf{(S1)} and \textbf{(S2)} in \Cref{Appendix C} are satisfied, consider the recursion \eqref{recursion 5} and \eqref{recursion 6}. For any uniformly pseudo-Lipschitz sequence of functions $\phi_n: \mathbb{R}^{n} \times \mathbb{R}^n \rightarrow \mathbb{R}$ and $\varphi_p: \mathbb{R}^{p} \times \mathbb{R}^p \rightarrow \mathbb{R}$, 
\begin{align}
    \phi_n (\mathbf{z}^t, \mathbf{w}) \overset{\mathbb{P}}{\rightarrow} \mathbb{E}\left[ \phi_n (\mathbf{w} + \sqrt{ \tau_t^2 - \sigma_w^2} \mathbf{Z}', \mathbf{w})\right] \\
    \varphi_p (\bm \beta^{t} + \mathbf{X}^\top \mathbf{z}^t, \Pi ) \overset{\mathbb{P}}{\rightarrow} \mathbb{E} \left[\varphi_p (\Pi + \tau_t \mathbf{Z}, \Pi )  \right]
\end{align}
where $\tau_t$ is defined in \eqref{state evolution}, $\mathbf{Z}'\sim \mathcal{N}(\mathbf{0}, \mathcal{I}_n)$ and $\mathbf{Z}\sim \mathcal{N}(\mathbf{0}, \mathcal{I}_p)$.
\end{lemma}

In summary, to see that Theorem 1 holds, 
we obtain that $\bm \beta^t+\bm X^\top\bm z^t\approx\mathbf{\Pi}+\tau_t\bm Z$ from Lemma \ref{berthier}. Together with $\bm{\beta}^{t+1} = \eta_{\gamma}(\mathbf{X}^\top \mathbf{z}^t + \bm{\beta}^t, \alpha \tau_t)$ from \eqref{recursion 5}, we have $\bm\beta^t\approx\eta_{\gamma}(\mathbf{\Pi}+\tau_t\bm Z,\alpha\tau_t)$. Finally Theorem \ref{theorem2} and  to obtain that $\bm {\hat\beta}=\eta_\gamma(\bm \beta^{t} + \mathbf{X}^\top \mathbf{z}^t, \alpha \tau_t)$ within uniformly pseudo-Lipschitz loss, in the large system limit.

\subsection{Beyond Pseudo-Lipschitz Functions: TPP and FDP of SGL}
While Theorem \ref{theorem1} only works under the pseudo-Lipschitz functions, which include the MSE, we can show that the result extends to important non-pseudo-Lipschitz functions such as TPP and FDP.

\begin{corollary}
Under assumptions (\textbf{A1}) - (\textbf{A5}), denote $\Pi^*=\Pi|\Pi\neq 0$ as the distribution of non-zero signals, the Sparse Group LASSO estimator satisfies
$$
\begin{aligned}
\operatorname{TPP}(\boldsymbol{\beta}, \boldsymbol{\lambda}):=\frac{\left|\left\{j: \widehat{\beta}_{j} \neq 0, \beta_{j} \neq 0\right\}\right|}{\left|\left\{j: \beta_{j} \neq 0\right\}\right|}
\stackrel{P}{\rightarrow} &\operatorname{TPP}^{\infty}(\Pi, \Lambda):=\mathbb{P} (|\Pi^* + \tau_* Z| > \gamma \alpha \tau_*),
\\
\operatorname{FDP}(\boldsymbol{\beta}, \boldsymbol{\lambda}):=\frac{\left|\left\{j: \widehat{\beta}_{j} \neq 0, \beta_{j}=0\right\}\right|}{\left|\left\{j: \widehat{\beta}_{j} \neq 0\right\}\right|}
\stackrel{P}{\rightarrow}& \operatorname{FDP}^{\infty}(\Pi, \Lambda):= \mathbb{P}\left(\Pi=0\Big||\Pi + \tau_* Z| > \gamma \alpha \tau_*\right)
\\
=&\frac{2(1-\mathbb{P}(\Pi\neq 0))\Phi(-\gamma \alpha)}{2(1-\mathbb{P}(\Pi\neq 0))\Phi(-\gamma \alpha) + \mathbb{P}(\Pi\neq 0)\operatorname{TPP}^{\infty}(\Pi, \Lambda)},
\end{aligned}
$$
where the convergence is in probability and $(\tau_*,\alpha)$ is the unique solution to the state evolution \eqref{state evolution} and calibration \eqref{calibration}.
\end{corollary}
At high level, although TPP and FDP are not pseudo-Lipschitz, we can approximate them by a sequence of pseudo-Lipschitz functions as $n,p\to\infty$. We omit the proof here as it is extremely similar to the LASSO case by \cite[Theorem B.1]{bogdan2013statistical}.

After establishing how we can use the AMP theory to characterize some statistical quantities, we now introduce the definition of `\textbf{SGL path}' for future study: for a fixed $\gamma$, the SGL path is the space of all SGL solutions $\hat{\bm\beta}(\lambda)$ as $\lambda$ varies in $(0,\infty)$. This notion allows us to understand explicitly how the penalty $\lambda$ affects the performance measures, as well as some interesting trade-offs. For example, one may derive the trade-off between TPP and FDP of SGL, similar to that of the LASSO \cite{su2017false} and that of the SLOPE \cite{bu2021characterizing}, where improving TPP of an estimator always results in worsening FDP as well. Mathematically speaking, the trade-off studies for a given TPP level $u$, the minimum FDP achievable as $\inf _{(\Pi, \Lambda): \operatorname{TPP}^{\infty}(\Pi, \Lambda)=u} \operatorname{FDP}^{\infty}(\Pi, \Lambda)$.

From now on, our simulations will focus on the MSE, TPP (or power) and FDP estimated by AMP instead of the empirical values.

\begin{figure}[ht!]
\centering
\hspace{-5.5cm}
\subfigure[]{\includegraphics[width=3.8cm,height=2.8cm]{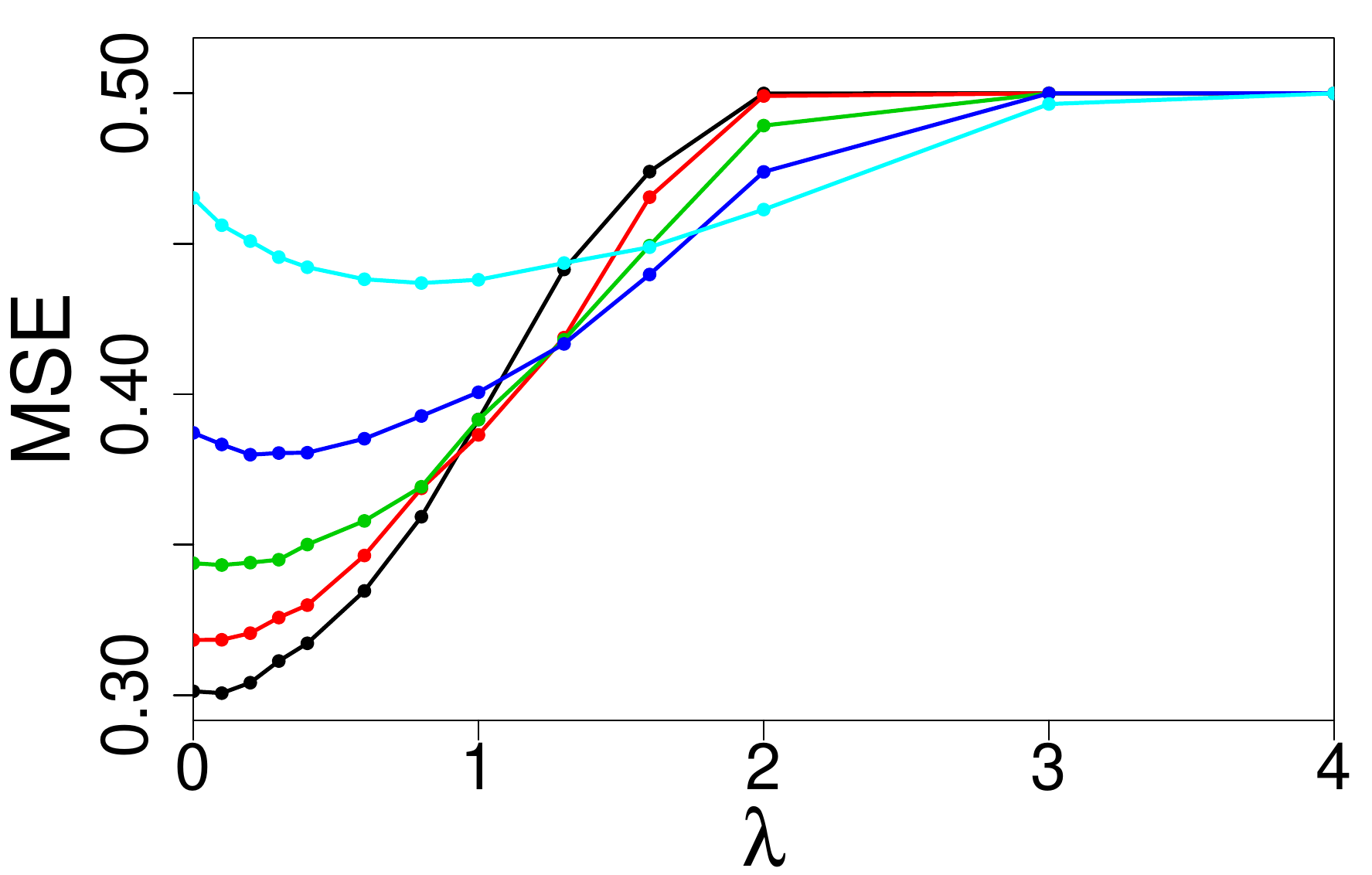}}
\hspace{-0.2cm}
\subfigure[]{\includegraphics[width=3.8cm,height=2.8cm]{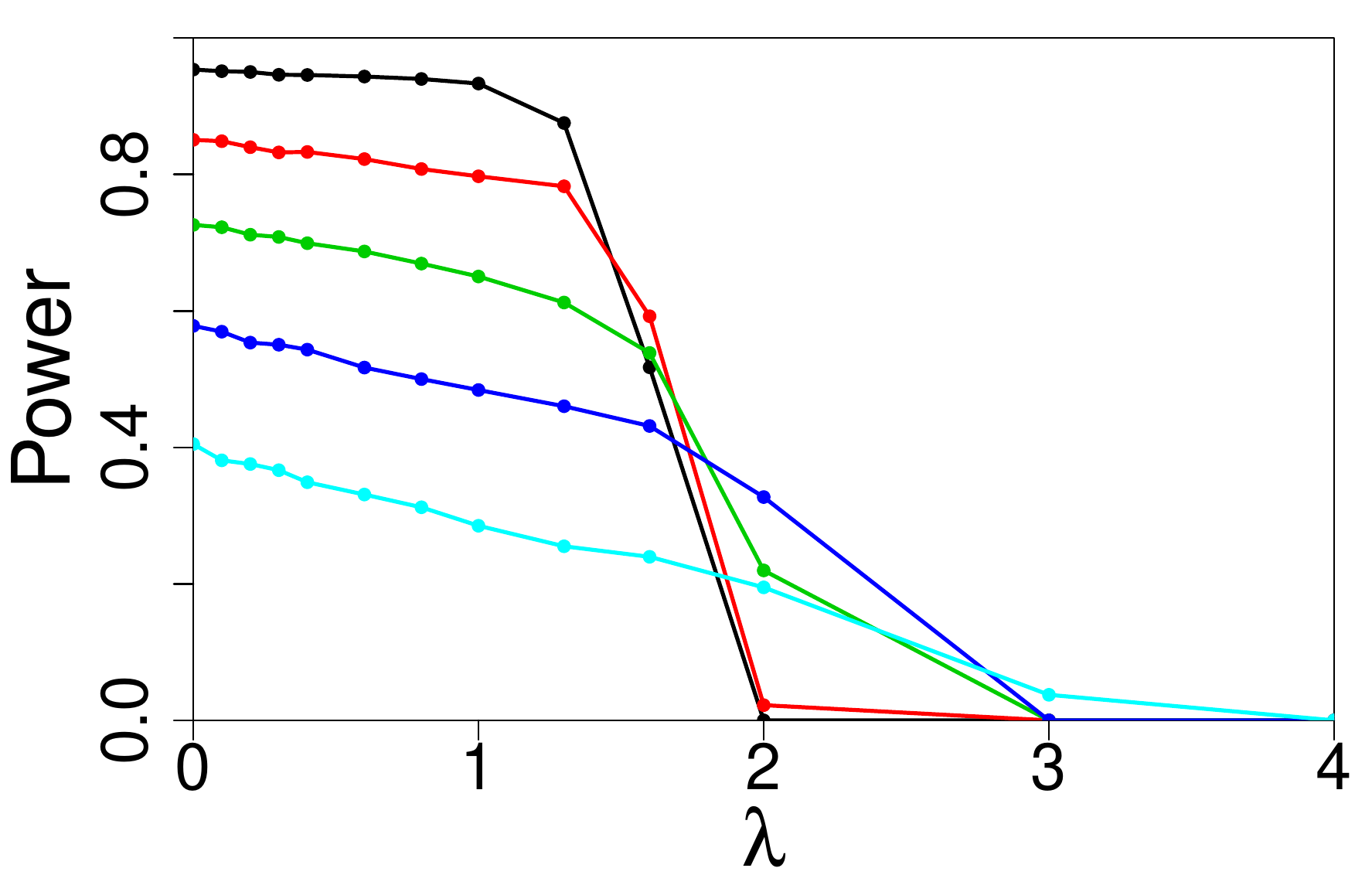}}
\hspace{-0.2cm}
\subfigure[]{\includegraphics[width=3.8cm,height=2.8cm]{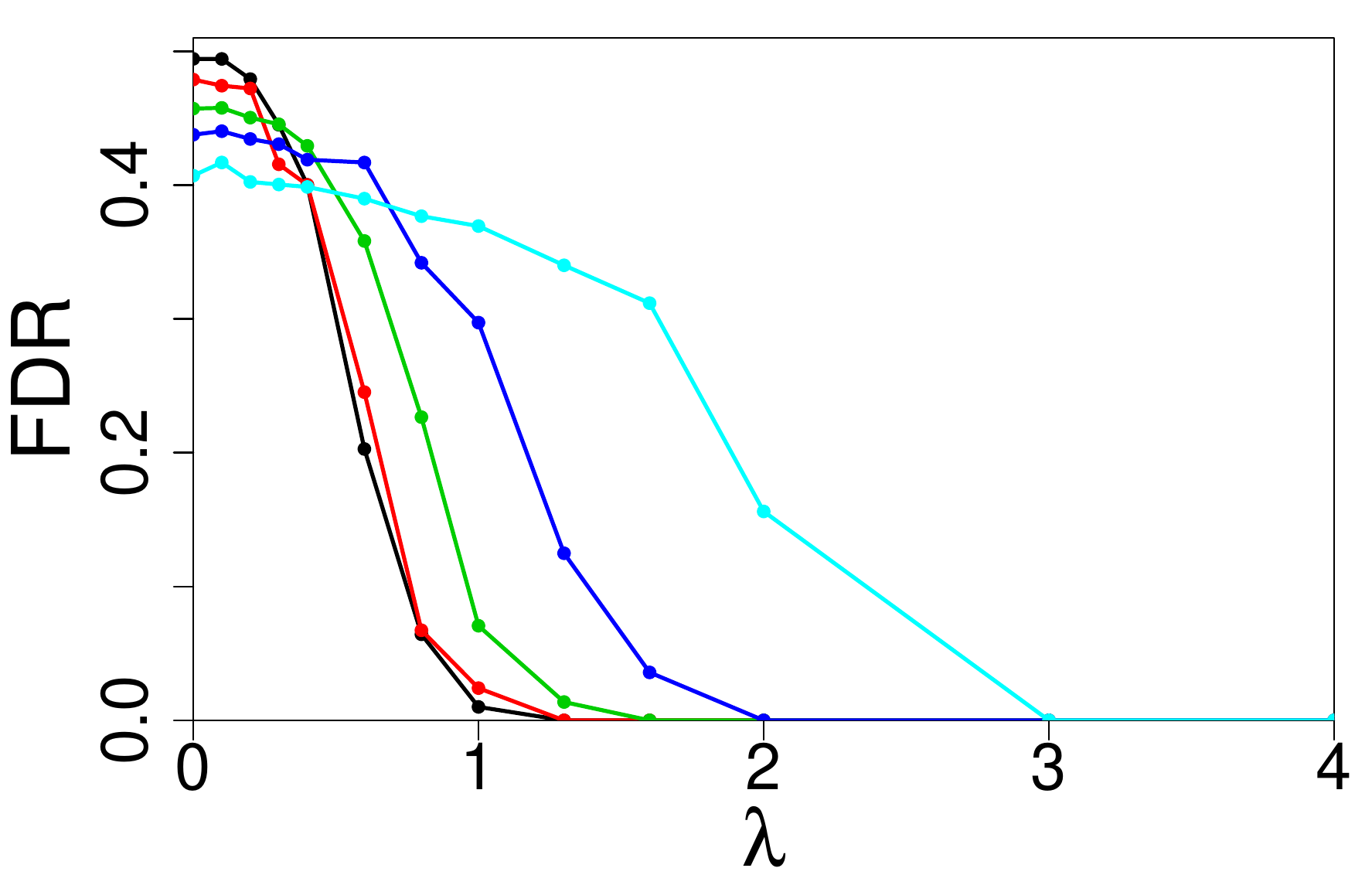}}
\hspace{-0.2cm}
\subfigure[]{\includegraphics[width=3.8cm,height=2.8cm]{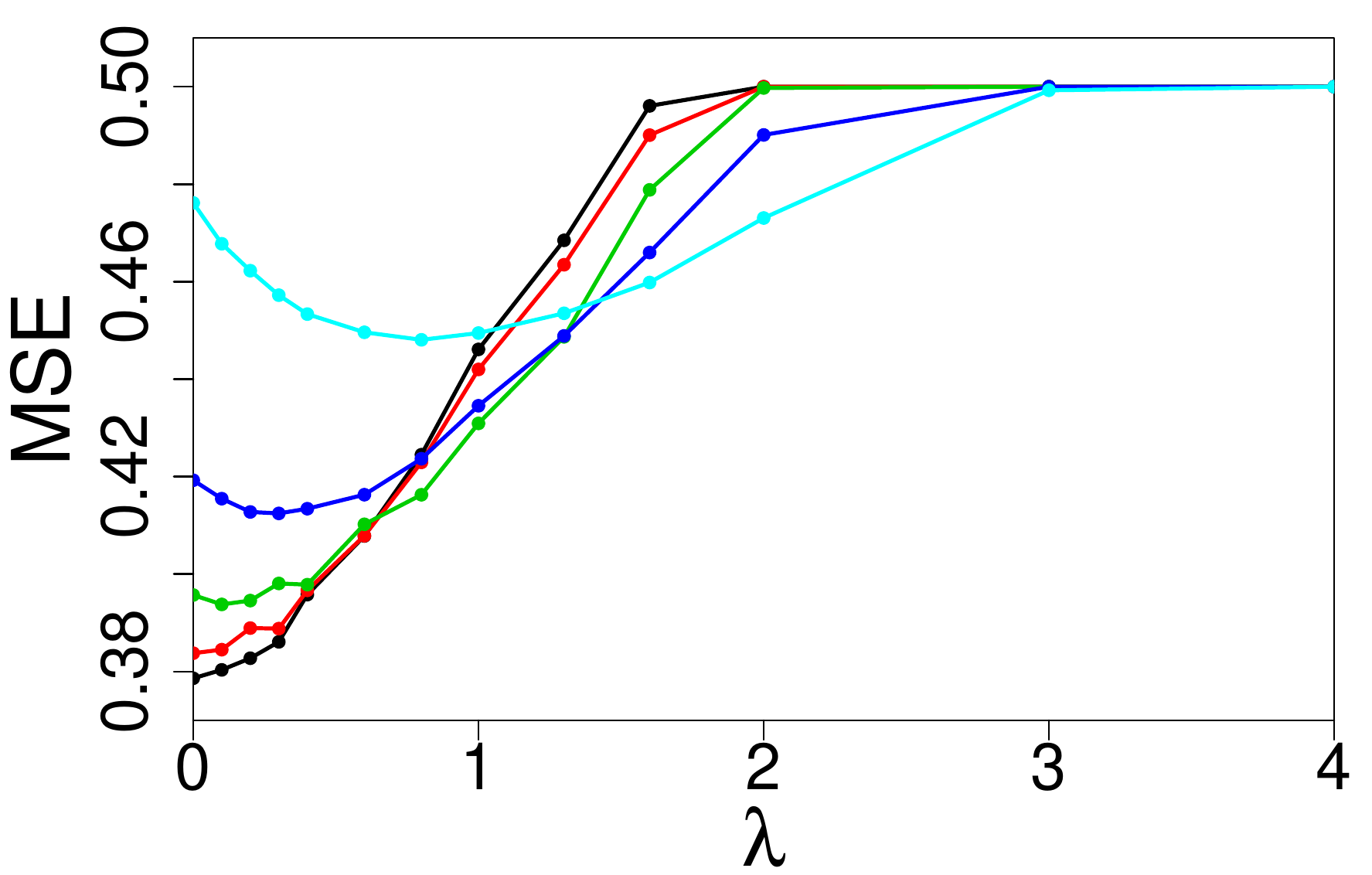}}
\hspace{-0.3cm}
\includegraphics[width=1.5cm]{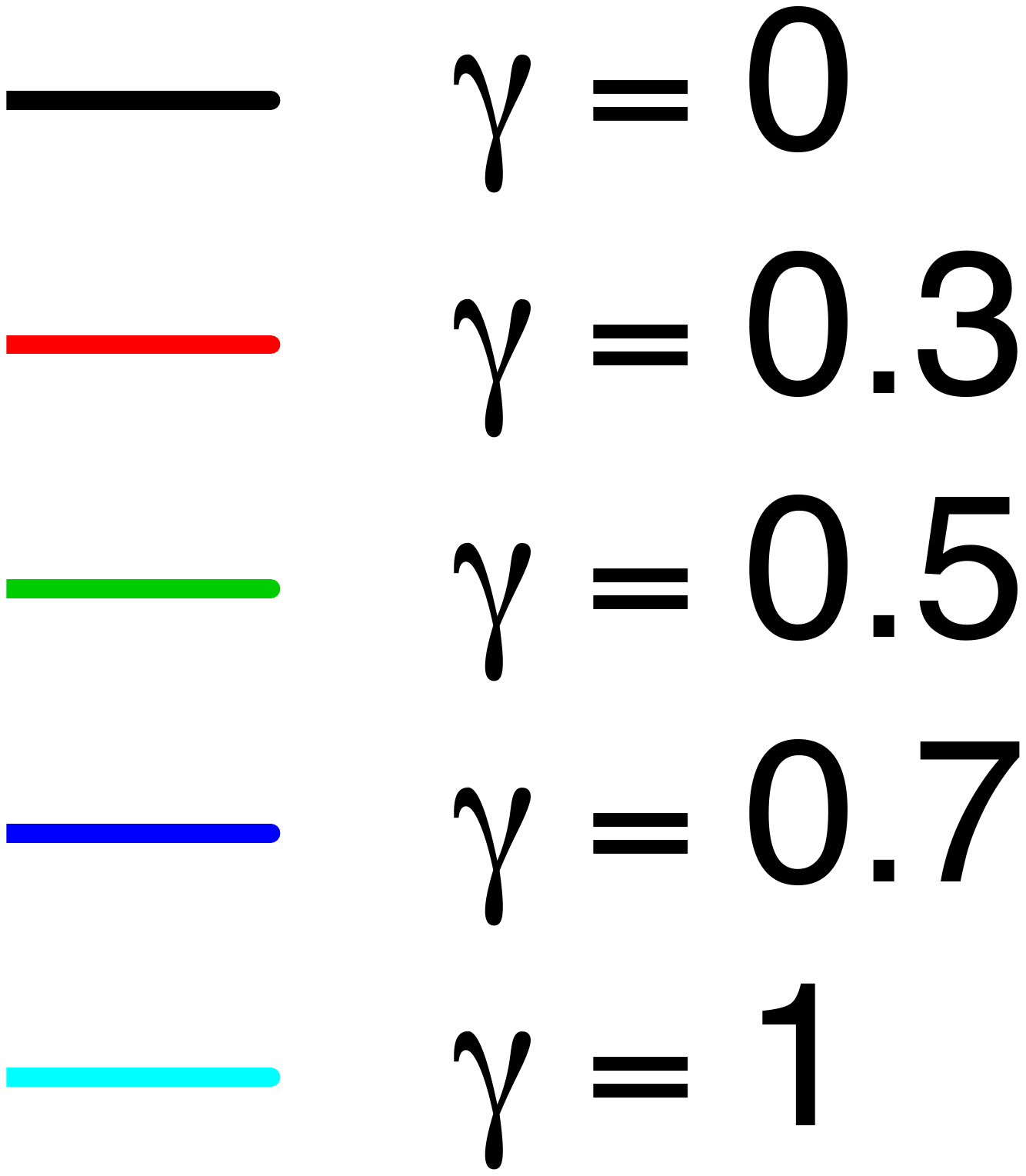}
\hspace{-6cm}
\vspace{-0.5cm}
\caption{(a)(b)(c) MSE/Power (TPP)/FDP on SGL path given perfect group information; (d) MSE on SGL path given mixed group information.}
\label{f4}
\end{figure}

%% file: File/004Simulation.tex
\section{Simulation}

Throughout this section, unless specified otherwise, the default setting is as follows: $\delta:=n/p=0.25$, $\bm X$ has i.i.d. entries from $\mathcal{N}(0,1/n)$, $\gamma=0.5$, $\sigma_{\bm w}=0$ and the prior is Bernoulli($0.5$). We consider at most two groups and set the perfect group information as default. The group information $\bm g$ is perfect if all the true signals are classified into one group, while the other group contains all the null signals. In contrast, the mixed group information is the case when only one group exists, so that the true and null signals are mixed.

\subsection{State Evolution Characterization}
First we demonstrate that AMP state evolution indeed characterizes the solution $\bm{\hat\beta}$ by $\eta_\gamma(\bm\beta+\tau \bm Z,\alpha\tau)$ asymptotically accurately. Figure \ref{fig:mse_high}(a) clearly visualizes and confirms Theorem \ref{theorem1} by the distributional similarity. In Figure \ref{fig:mse_high}(b), the empirical MSE $\|\bm{\hat\beta}-\bm\beta\|^2/p$ has a mean close to the AMP estimate $\delta(\tau^2-\sigma_{\bm w}^2)$ in Corollary \ref{cor:mse} and the variance decreases as the dimension increases.

\begin{figure}[h]
\centering
\hspace{-4.35cm}
\subfigure[]{\includegraphics[width=6cm,height=4cm]{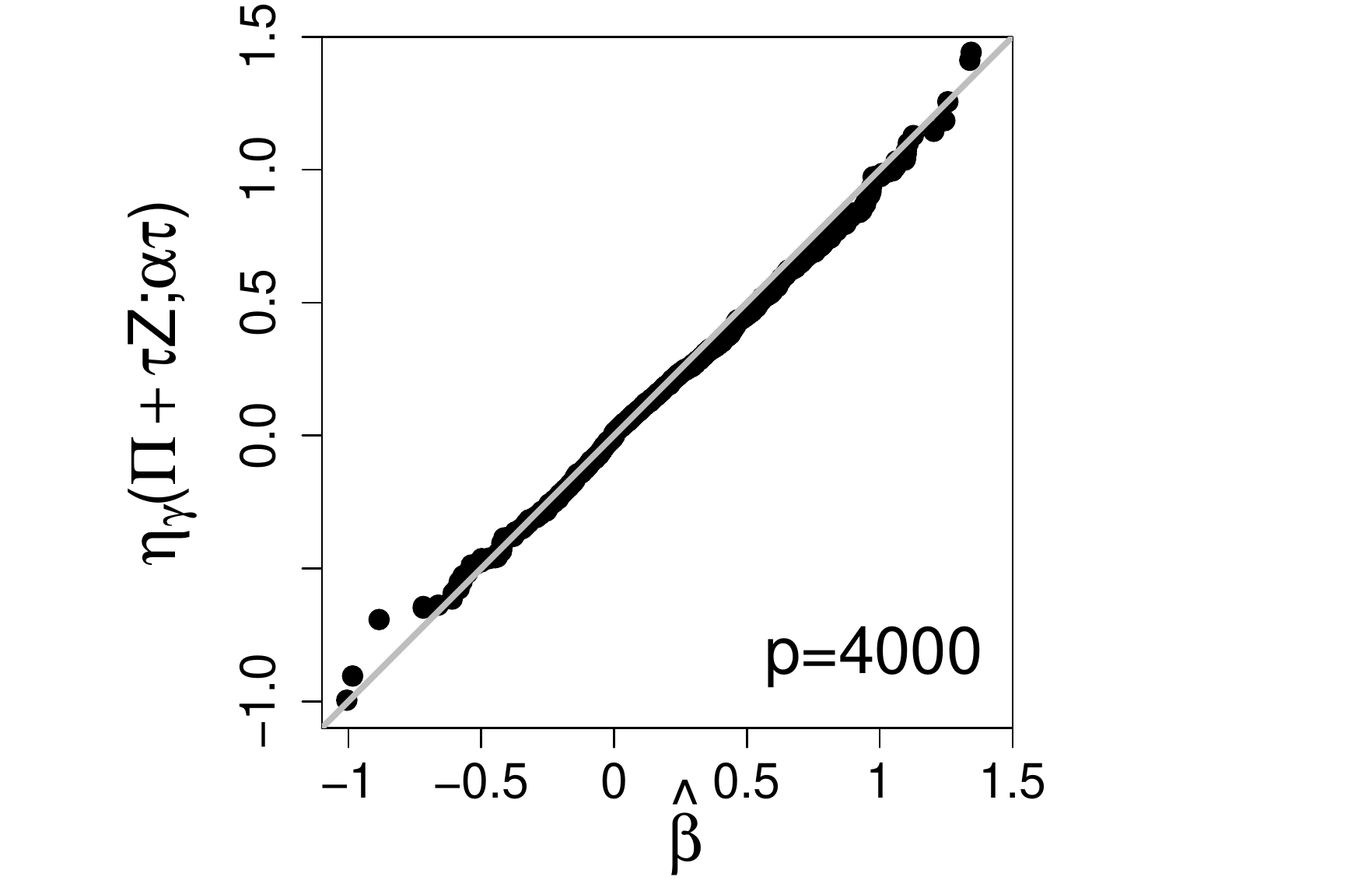}}
\hspace{-1.25cm}
\subfigure[]{\includegraphics[width=7cm,height=4cm]{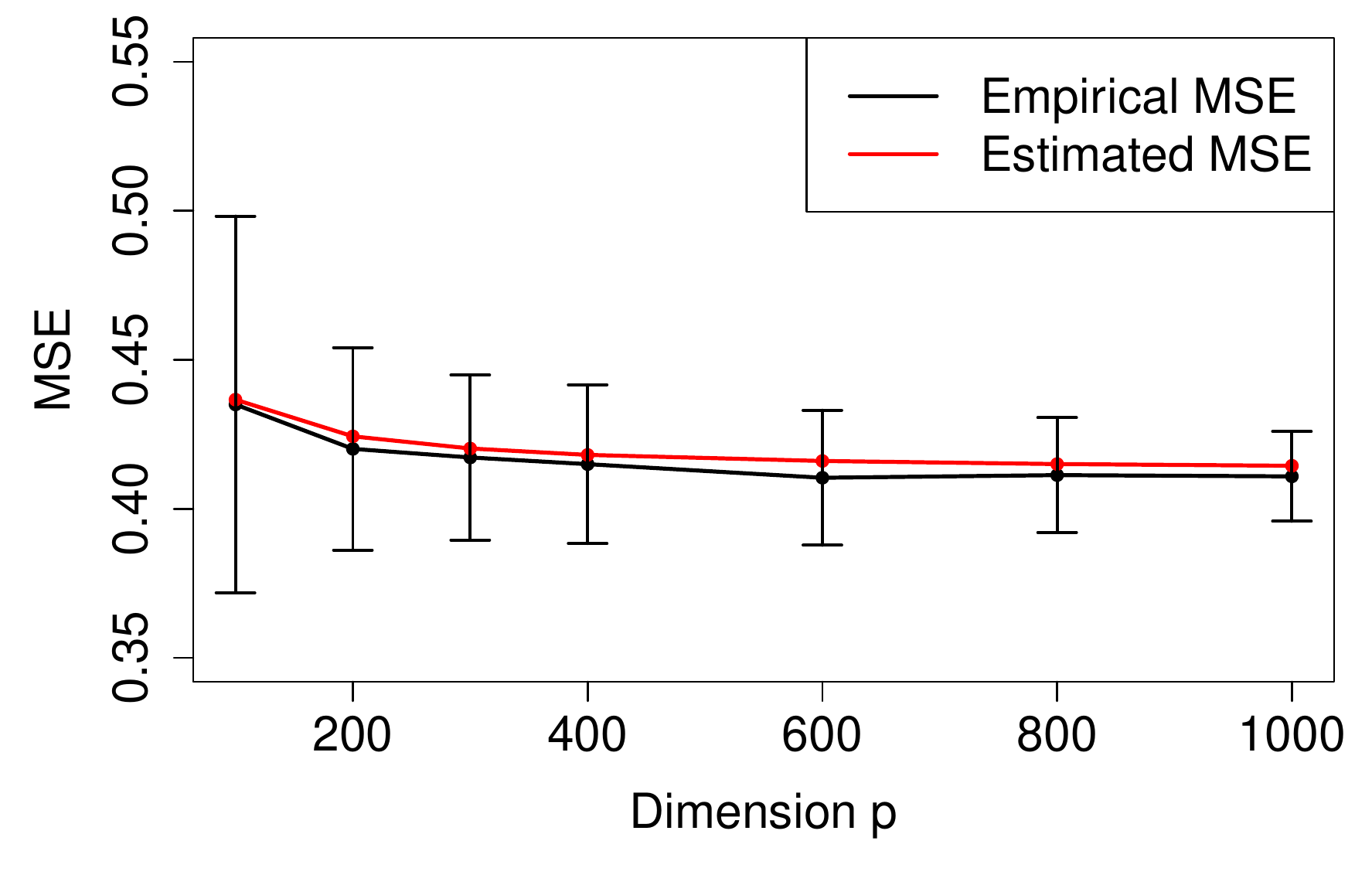}}
\hspace{-4.1cm}
\vspace{-0.5cm}
\caption{Illustration of Theorem \ref{theorem1}: AMP characterizes SGL solution. Notice $\epsilon=0.5,\delta=0.25,\sigma_w=1,\gamma=0.5$ and group information is perfect. (a) Quantile-quantile plot of $\bm{\hat\beta}$ and AMP solution with $\lambda=1$; (b) Empirical and estimated MSE by AMP over 100 independent runs with $\alpha=1$ ($\lambda=0.32$). The error bars correspond to one standard deviation.}
\label{fig:mse_high}
\end{figure}

\subsection{Benefits of Groups}

Now we investigate the benefits of groups under different scenarios. We set the dimension $p=400$ and the noise $\sigma_{\bm w}=0$. Figure
\ref{f4}(a) plots MSE against $\lambda$ given perfect group information, while Figure \ref{f4}(d) demonstrates the case with the mixed group information. Figure
\ref{f4}(b) and \ref{f4}(c) plot the power and FDR against $\lambda$ given the perfect group information.

We observe that, fixing models, better group information helps the models achieve better performances, especially when $\gamma$ is small, i.e. when SGL is closer to the Group LASSO. By comparing Figure \ref{f4}(a) and Figure \ref{f4}(d), we see an increase of MSE when signals are mixed by the group information. Somewhat surprisingly, even SGL with the mixed group information may achieve better MSE than LASSO, which does not use any group information.

On the other hand, for fixed group information, models with different $\gamma$ enjoy the benefit of group information differently: in Figure \ref{f4}(a), we notice that the performance depends on the penalty $\lambda$: if $\lambda$ is small enough, then SGL with smaller $\gamma$ performs better; if $\lambda$ is sufficiently large, then SGL with larger $\gamma$ may be favored. 

We compare SGL with other $\ell_1$-regularized models, namely the LASSO, Group LASSO, adaptive LASSO \cite{zou2006adaptive} and elasitc net \cite{zou2005regularization} in Figure \ref{fig:tradeoffs}, given perfect group information.
In the type I/II tradeoff, Group LASSO and SGL with $\gamma=0.5$ demonstrates dominating performance over other models. However, in terms of the estimation MSE (between $\bm{\hat\beta}$ and 
$\bm\beta$), SGL allows to achieve smaller MSE but selects more features. In both figures, SGL shows a piecewise continuous pattern and it would be interesting to derive the explicit form with AMP in the future.

\begin{figure}[!h]
    \centering
    \includegraphics[width=7cm]{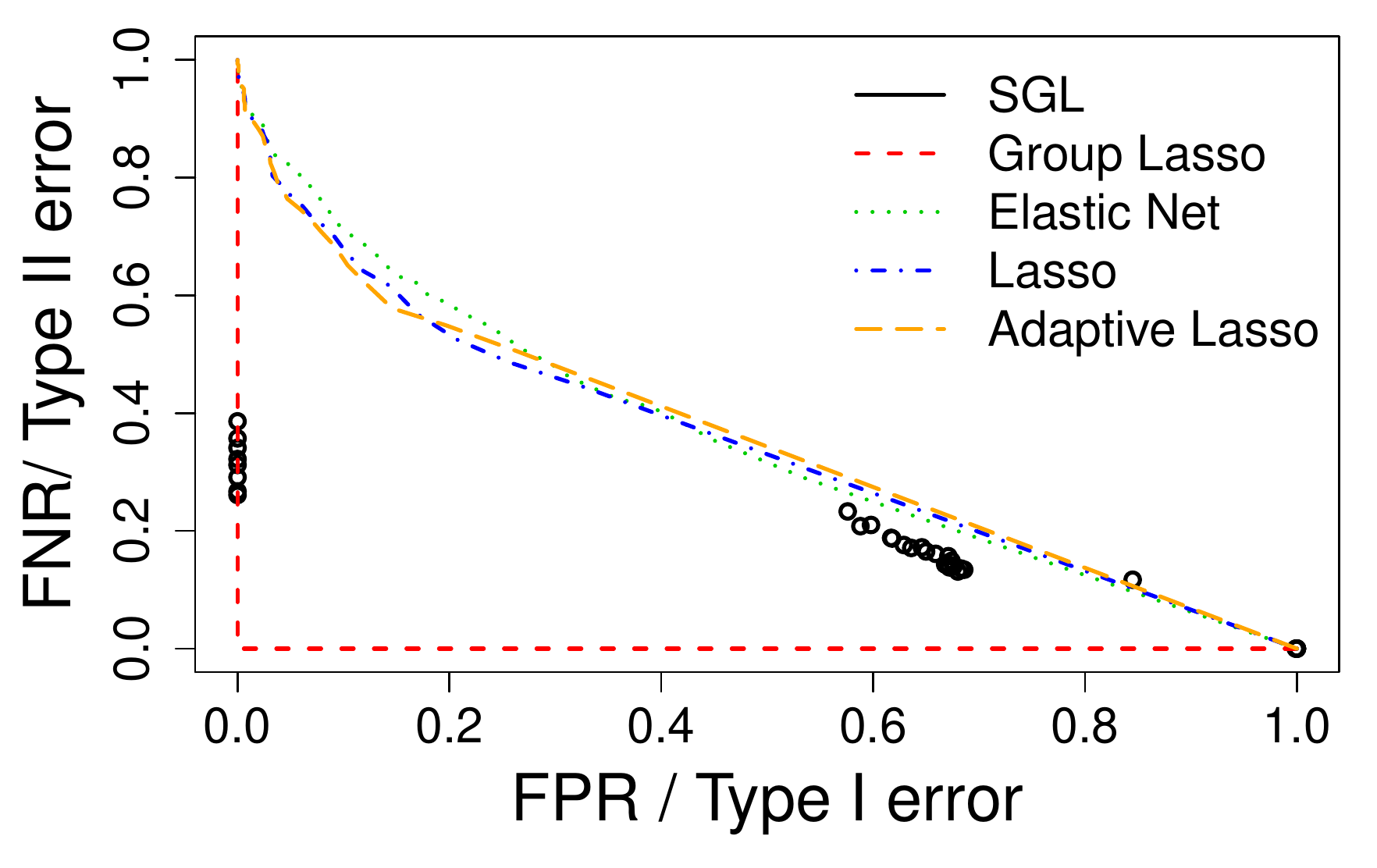}
    \includegraphics[width=7cm]{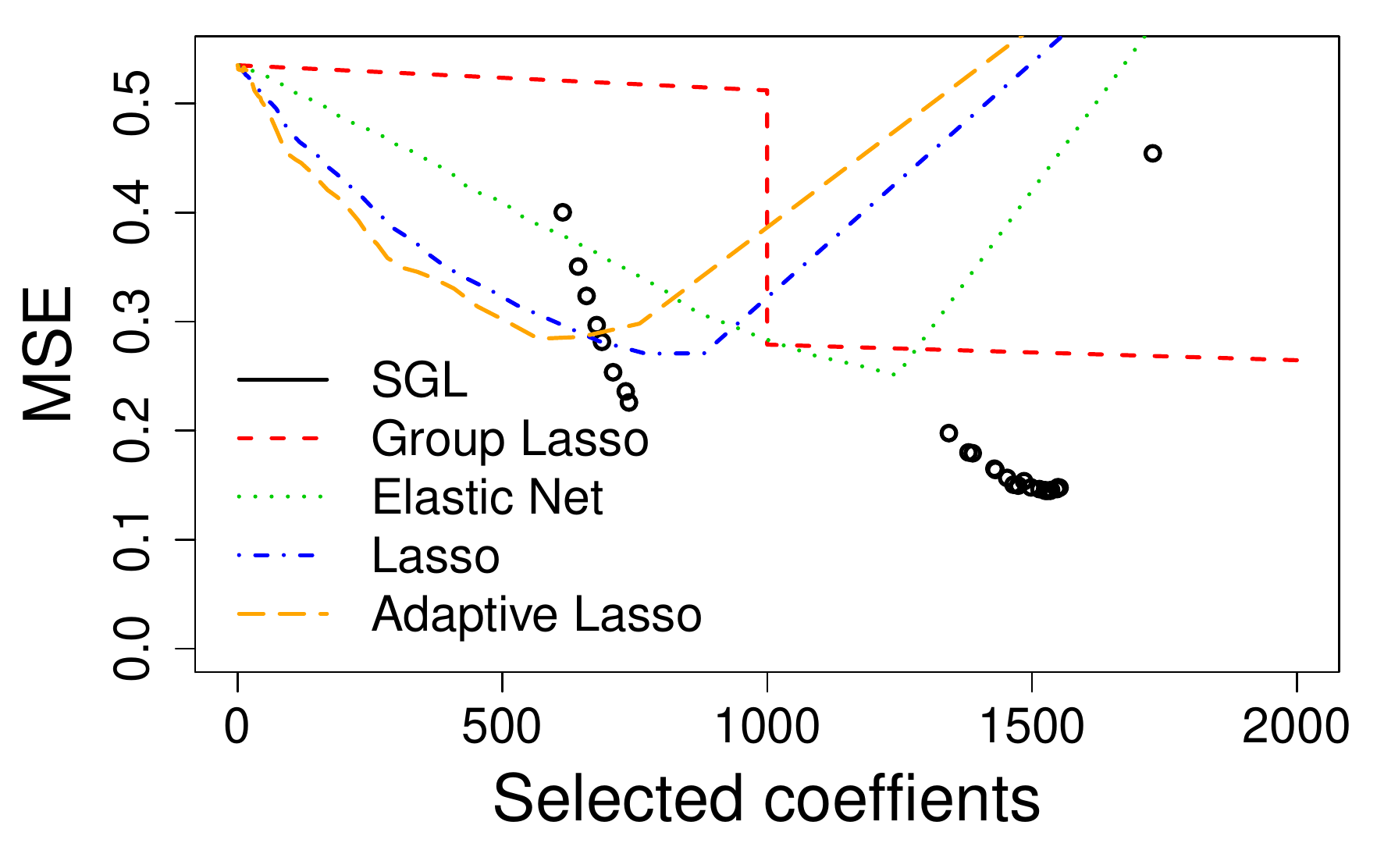}
    \vspace{-0.5cm}
    \caption{Left: Type I/II error tradeoff. Right: MSE against the number of selected coefficients with perfect group information. Here $\sigma_w=0, \bm X\in\mathbb{R}^{1000\times 2000}$ is i.i.d. Gaussian and the prior follows a Bernoulli-Gaussian with 0.5 probability being standard Gaussian.}
    \label{fig:tradeoffs}
\end{figure}

We further compare SGL AMP to the MMSE AMP (or Bayes-AMP) \cite{donoho2010message}, which by design finds the minimum MSE over a wide class of convex regularized estimators \cite{celentano2019fundamental}. In Figure \ref{fig:mmse}, we plot SGL with different $\gamma$ and carefully tune the penalty of each model to achieve its minimum MSE. We summarize that, empirically, SGL AMP with good group information is very competitive to MMSE AMP and smaller $\gamma$ leads to better performance. Nevertheless, this observed pattern may break if the group information is less correct.

\begin{figure}[!htb]
\centering
\includegraphics[width=0.5\linewidth]{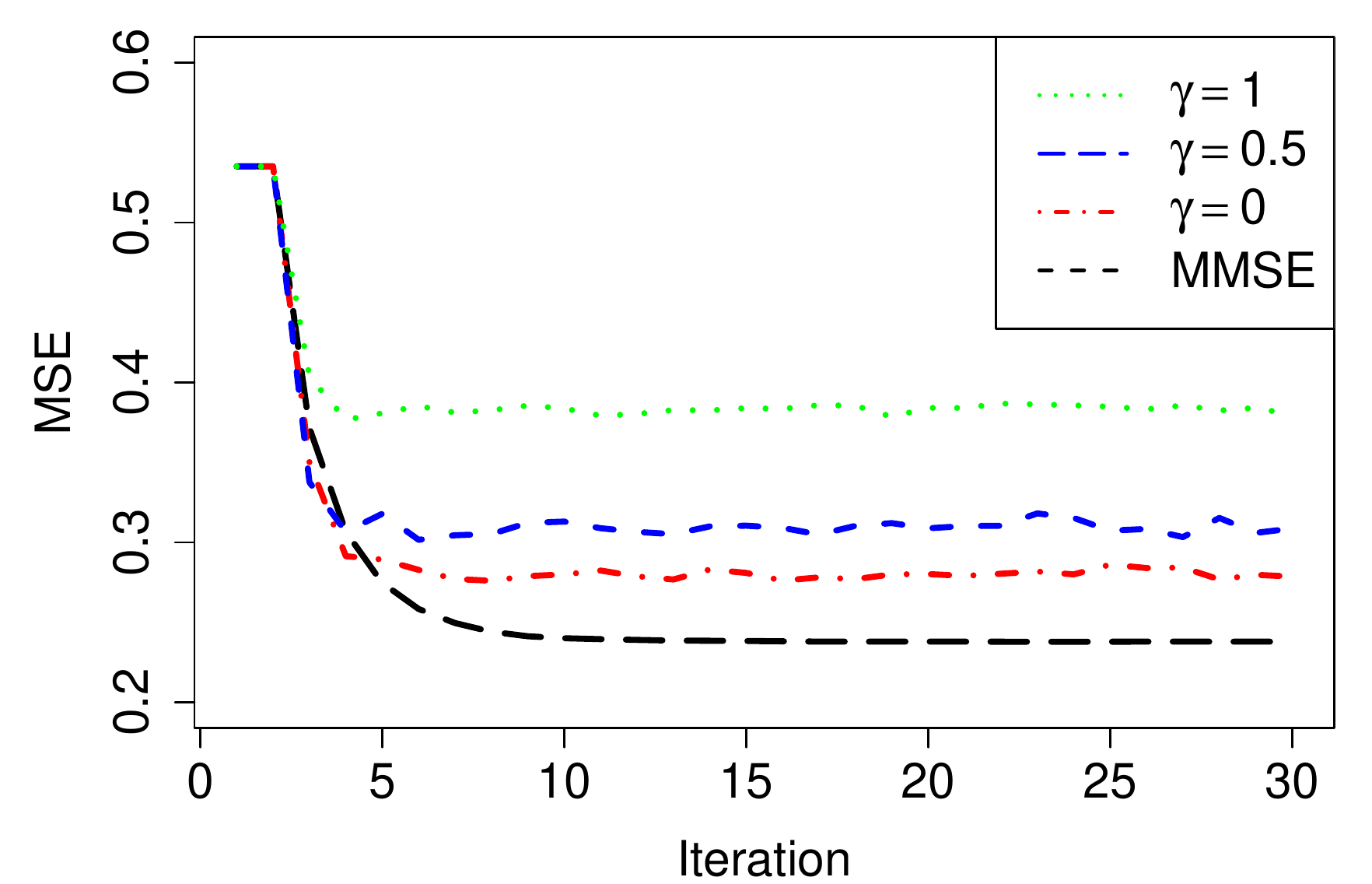}
\vspace{-0.3cm}
\caption{MSE of AMP algorithms with perfect group information. Here $n=1000, p=2000, \sigma_w=0, \bm X$ is i.i.d. Gaussian and the prior follows a Bernoulli-Gaussian with 0.5 probability being standard Gaussian, and being zero otherwise.}
\label{fig:mmse}
\end{figure}

\vspace{-0.2cm}
\subsection{Extensions of SGL AMP}
\label{extensions}
The theoretical result of vanilla AMP assumes that the design matrix $\bm X$ is i.i.d. Gaussian (or sub-Gaussian) \cite{bayati2011dynamics,bayati2015universality}. The convergence of AMP may be difficult if not impossible on the real-world data. Nevertheless, empirical results suggest that AMP works on a much broader class of data matrices even without theoretical guarantees. In our experiments, we observe that the performance of AMP is very similar to Figure \ref{fig:MSE1} for i.i.d. non-Gaussian data matrices (c.f. Appendix\ref{app:extension}).

\begin{figure}[!htb] 
\centering
\includegraphics[width=0.5\linewidth]{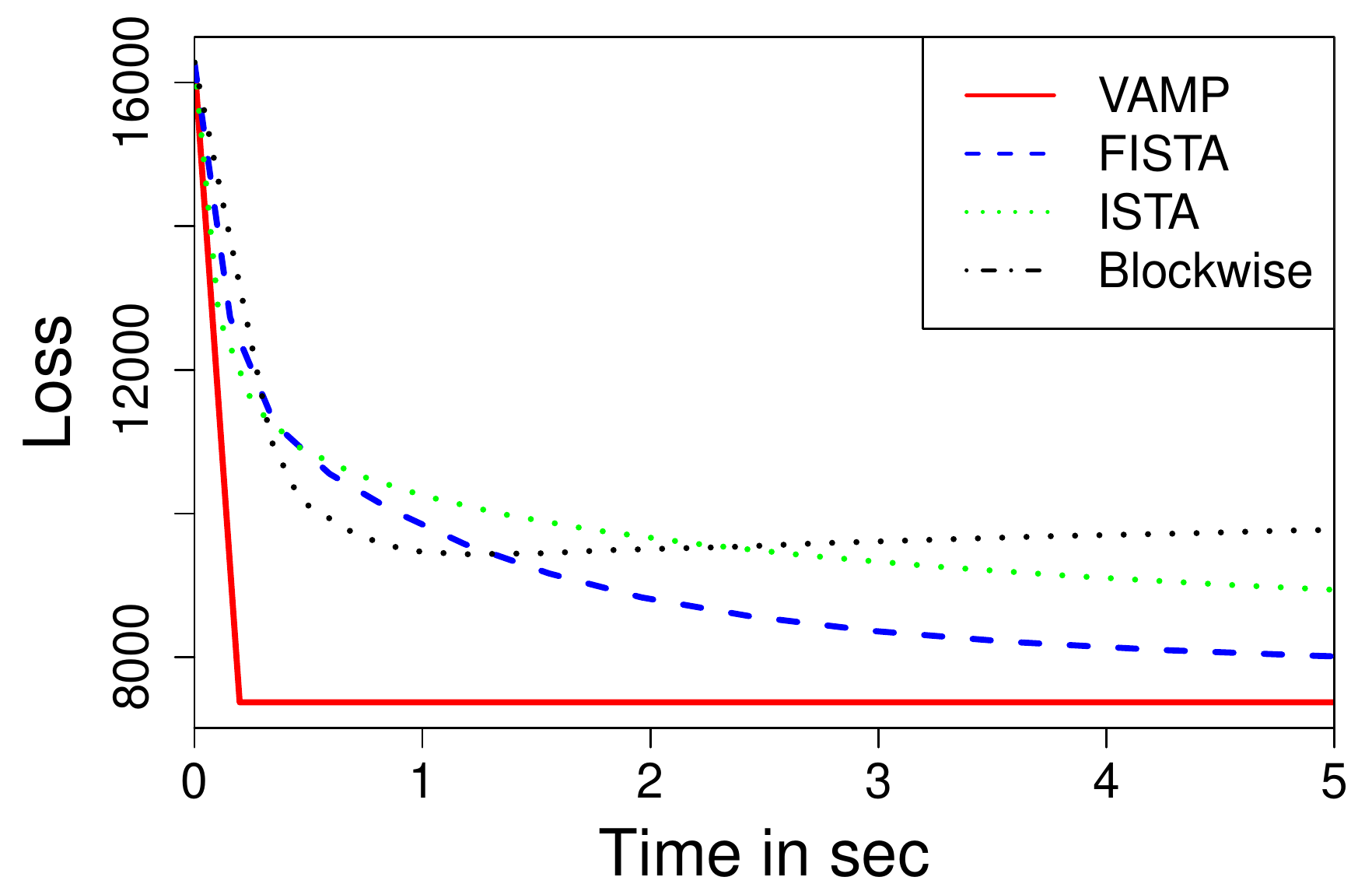}
\vspace{-0.3cm}
\caption{VAMP algorithm on the Adult income dataset with $\gamma = 0.5,  \bm g=(1,\cdots,1)$, $\lambda = 5$ and damping constant $\mathcal{D}=0.1$.}
\label{fig:vamp}
\end{figure}

On the other hand, we may relax the assumption of `i.i.d.' by leveraging a variant of AMP, called vector-AMP or VAMP \cite{rangan2019vector}. It has been rigorously shown that VAMP works on a larger class of data matrices, i.e. the right rotationally-invariant matrices. We emphasize that applying VAMP to non-separable penalties is in general an open problem, though there has been some progress for certain specific type of non-separability \cite{manoel2018approximate}. In Appendix \ref{app:extension}, we substitute the soft-thresholding function with the SGL proximal operator to extend the LASSO VAMP to the SGL VAMP. We implement our SGL VAMP (Algorithm \ref{alg:vamp}) on the Adult Data Set \cite{kohavi1996scaling}, which contains 32,561 samples and 124 features to predict an adult's income influenced by the individual’s education level, age, gender, occupation, and etc. We observe that on this specific dataset, SGL VAMP converges in one single iteration and shows its potential to work on other real datasets.



%% file: File/005Conclusion_Future_Work.tex
\section{Discussion and Future Work}
\label{Conclusion}
In this work, we develop the approximate message passing algorithm for Sparse Group LASSO via the proximal operator. We demonstrate that the proximal approach, including AMP, ISTA and FISTA, is more unified and efficient than the blockwise approach in solving high-dimensional linear regression problems. The key to the acceleration of convergence is two-fold. On one hand, by employing the proximal operator, we can update the estimation of the SGL minimizer for each group independently and simultaneously. On the other hand, AMP has an extra `Onsager reaction term’, $\langle \eta_{\gamma,\bm g}'(\mathbf{X}^*\mathbf{z}^t + \bm{\beta}^t, \alpha \tau_t) \rangle$, which corrects the algorithm at each step non-trivially. 

Our analysis of SGL AMP reveals some important results on the state evolution and the calibration. For example, the state evolution of SGL AMP only works on a bounded domain of $\alpha$, whereas in the LASSO case, the $\alpha$ domain is not bounded above and makes the penalty tuning more difficult. We then prove that SGL AMP converges to the true minimizer and characterizes the solution exactly, in an asymptotic sense. We highlight that such characterization is empirically accurate in the finite sample scenario and allows us to analyze certain statistical quantities of the SGL solution closely, such as $\ell_2$ risk, type I/II errors as well as the effect of the group information on these measures.

Our work suggests several possible future research. In one direction, it is promising to extend the proximal algorithms (especially AMP) to a broader class of models with structured sparsity, such as the sparse linear regression with overlapping groups, Group SLOPE and the sparse group logistic regression. On a different road, although AMP is robust in distributional assumptions in the sense of fast convergence under i.i.d. non-Gaussian measurements, multiple variants of AMP may be applied to adapt to real-world data. To name a few, one may look into SURE-AMP \cite{guo2015near}, EM-AMP \cite{vila2011expectation,vila2013expectation} and VAMP \cite{rangan2019vector} to relax the known signal assumption and non-i.i.d. measurement assumption.

%% file: File/006Appendix.tex
\appendix
\section{SGL Proximal Operator} \label{Appendix A}
\label{app:operator}
In this appendix we derive the SGL proximal operator rigorously, with an emphasis of the close connection between the soft-thresholding and the SGL proximal operator. Recall that the cost function is
\begin{equation} \label{cost}
\mathcal{C}(s,\bm \beta) = \frac{1}{2}\|s -  \bm \beta \|_2^2 + (1 - \gamma)\lambda \sum_{l=1}^L \sqrt{p_l} \|\bm \beta_l\|_2 +  \gamma \lambda \|\bm \beta\|_1.
\end{equation}
This can also be written as the following
\begin{align*}
    \sum_{l = 1}^L \left( \frac{1}{2}\|s_l -  \bm \beta_l \|_2^2 + (1 - \gamma)\lambda  \sqrt{p_l} \|\bm \beta_l\|_2 +  \gamma \lambda \|\bm \beta_l\|_1 \right).
\end{align*}
We apply the Proposition 1 in Appendix D.1 from \cite{shi2016primer} below 
\begin{proposition}[\cite{shi2016primer}]
Let $\bm x \in \mathbb{R}^n$. If $f(\bm x)$ is a convex, homogeneous function of order 1 (i.e.,$f(\alpha \bm x) = \alpha f(\bm x)$ for $\alpha > 0$) and $g(\bm x) = \beta \|\bm x\|_2$, then $\text{prox}_{f + g} = \text{prox}_g \circ \text{prox}_f$. 
\end{proposition}
Since 
\begin{align*}
   \text{prox}_{f} (x,b) = \begin{cases} 
      x - b & x > b \\
       x + b & x < -b  \\
      0 & \text{otherwise}
   \end{cases}
\end{align*}
and
\begin{align*}
    \text{prox}_{g} (x,b) = \begin{cases} 
      (1 - \frac{b}{\|x \|_2}) x & \|x\|_2 > b \\
      0 & \text{otherwise}
   \end{cases},
\end{align*}
 we then obtain our proximal operator
\begin{align*}
    \eta_{\gamma}(\bm s, \lambda)^{(j)} = \begin{cases}
    \eta_{\text{soft}}(s^{(j)}, \gamma  \lambda)\left(1 - \frac{(1 - \gamma)\lambda \sqrt{p_{l_j}}}{ \| \eta_{\text{soft}}(\bm s_{l_j}, \gamma \lambda) \|_2 }\right) &  \|\eta_{\text{soft}}(s^{(j)}, \gamma  \lambda) \|_2 > (1-\gamma) \lambda \sqrt{p_{l_j}} \\
    0 & \text{otherwise}
    \end{cases} 
\end{align*}

\section{Proximal Methods Analysis} \label{Appendix B}
\label{app:proximal}

\begin{figure}[!htb]
\centering
\hspace{-0.5cm}
\subfigure[]{\includegraphics[width=5.4cm]{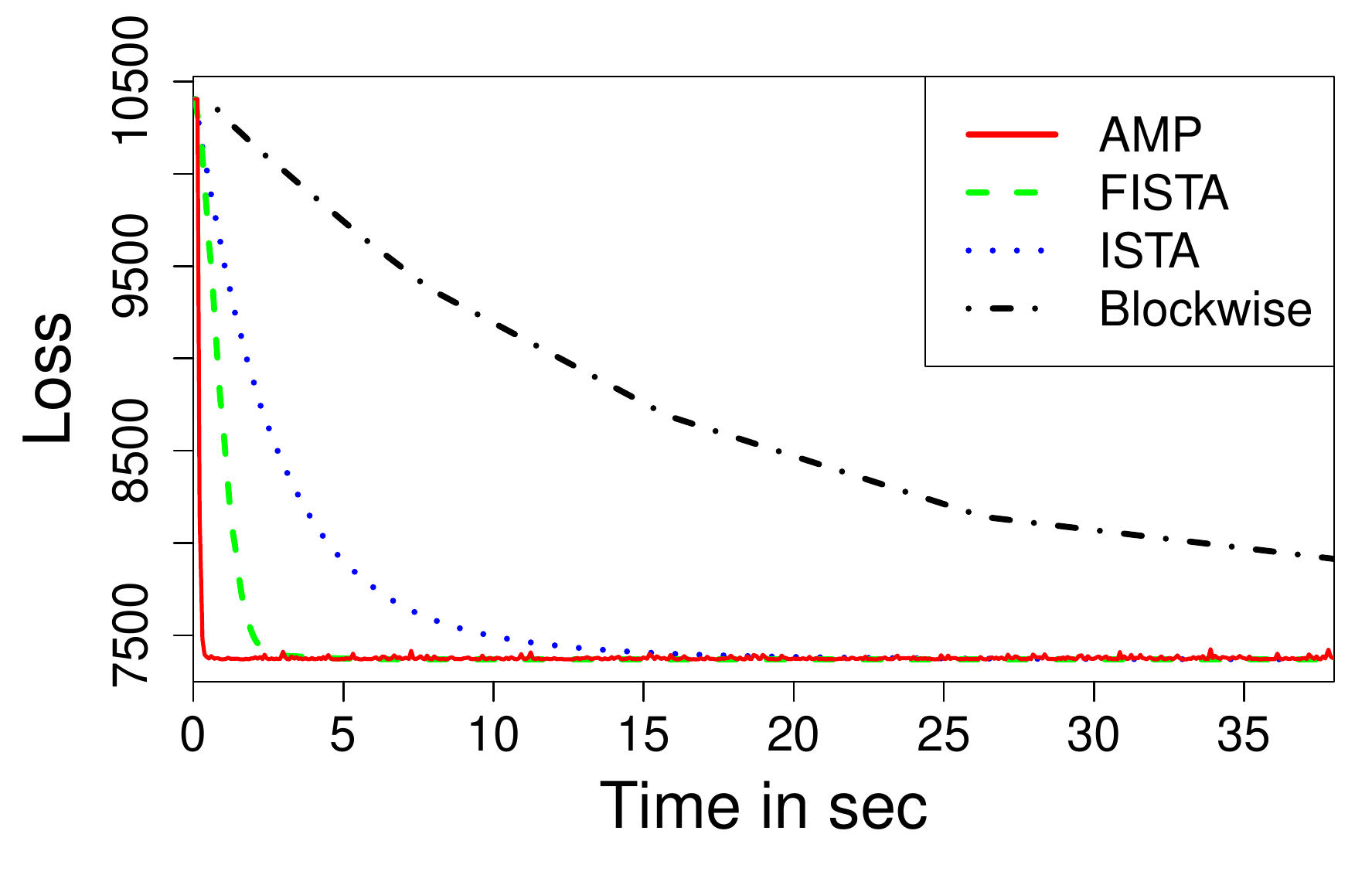}}
\hspace{-0.15cm}
\subfigure[]{\includegraphics[width=5.4cm]{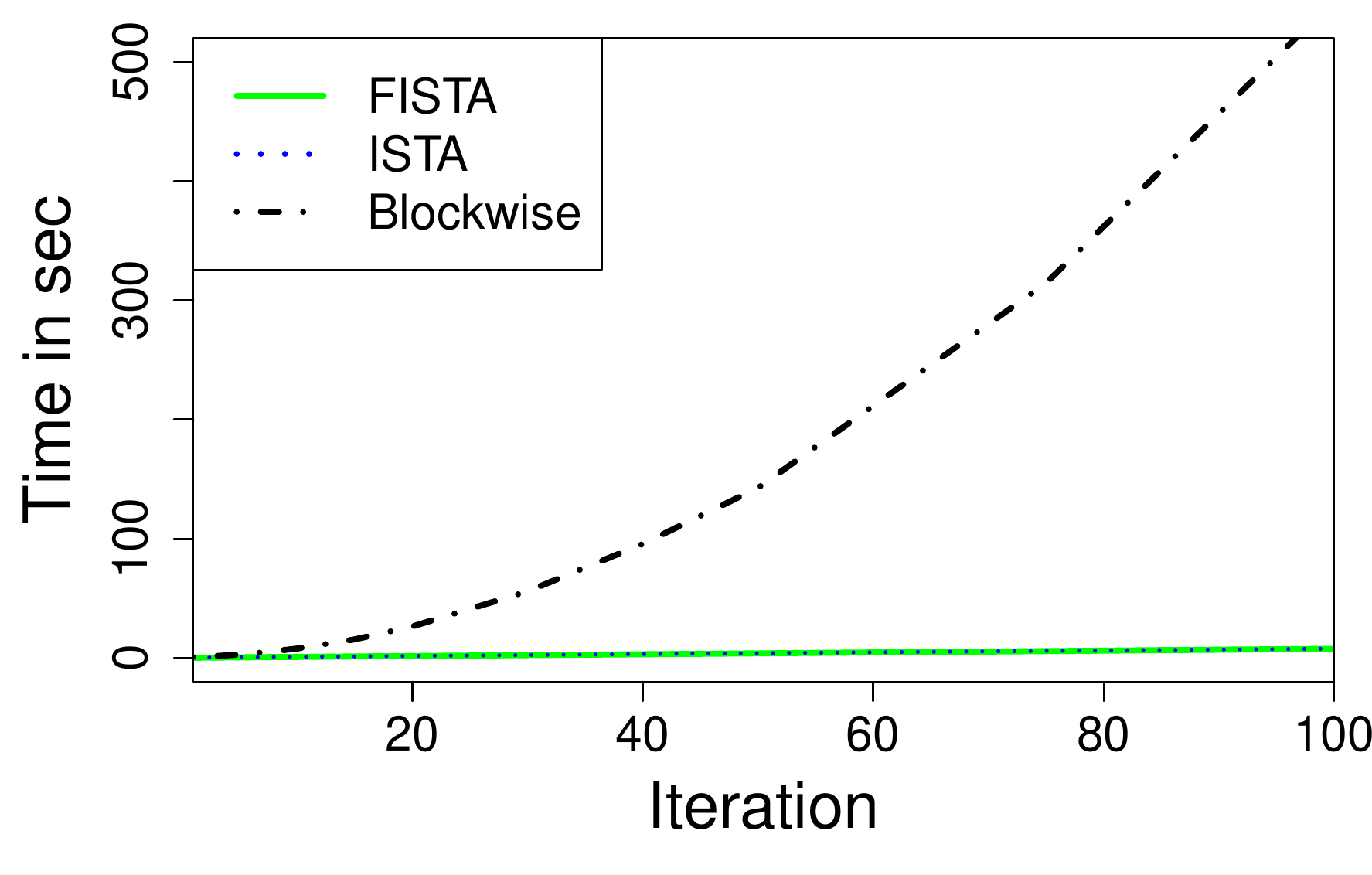}}
\hspace{-0.15cm}
\subfigure[]{\includegraphics[width=5.4cm]{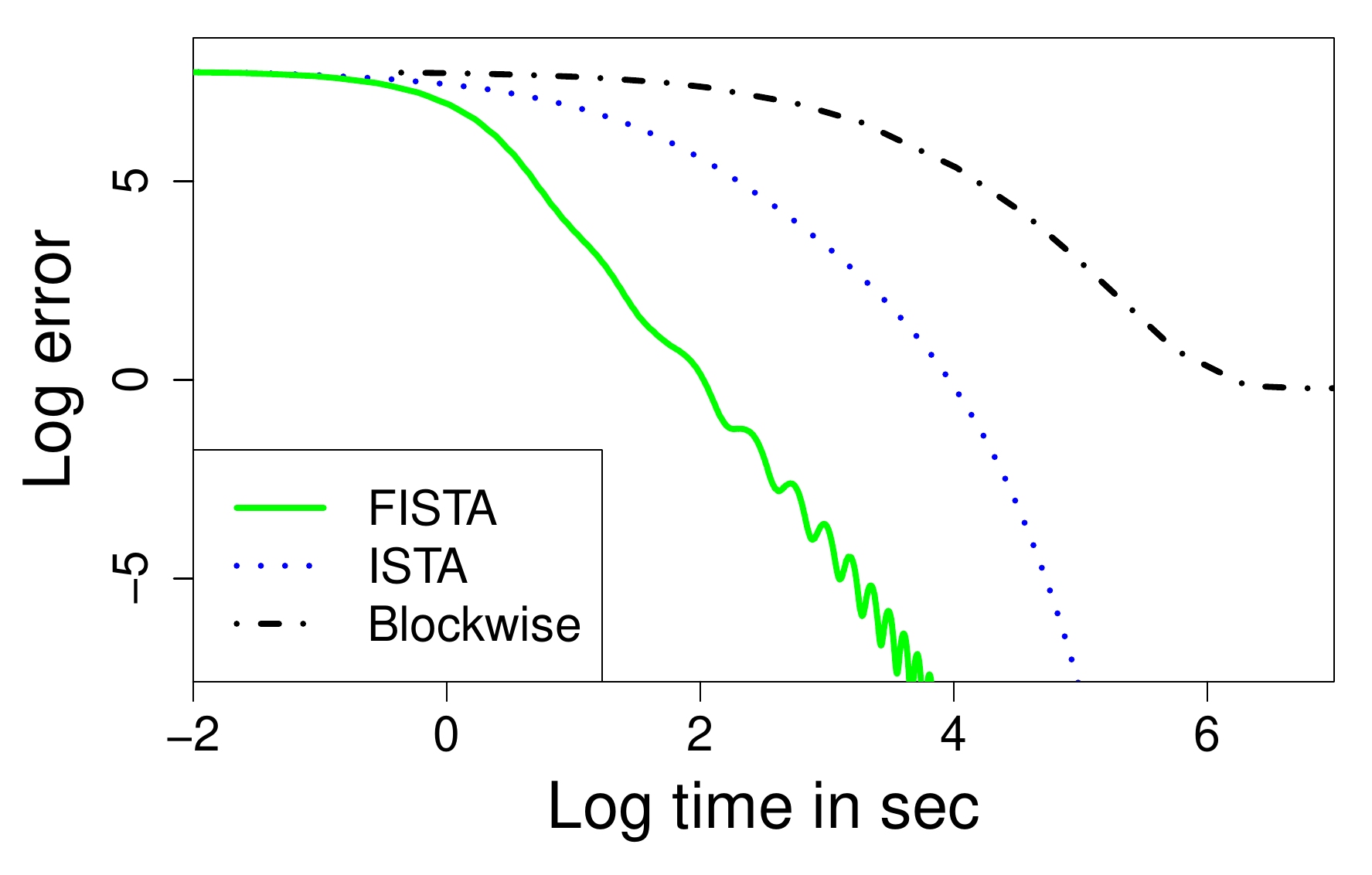}}
\hspace{-0.5cm}
    \vspace{-0cm}
    \caption{Comparison of proximal and blockwise methods by runtime. $p = 4000, n = 2000, \gamma = 0.5, \bm g=(1,\cdots,1)$, the entries of $\mathbf{X}$ are i.i.d. $\mathcal{N}(0,1/n)$, $\lambda = 1$, and the prior $\mathbf{\beta} $ is $5\times$Bernoulli(0.1).}
\label{fig:times}
\end{figure}

In this appendix, we analyze closely the proximal gradient descent (SGL ISTA) in comparison with the blockwise descent. We first state the two approaches for SGL. For the blockwise descent, we define $\bm r_{(-l)}:=\mathbf{y} - \sum_{k \neq l} \mathbf{X_k}\bm{\hat\beta}_k$ for group $l$.

\paragraph{ISTA (proximal gradient descent)}
\begin{align}
{\bm\beta}^{t+1}=\eta_\gamma({\bm\beta}^{t}+s \bm X^\top(\bm y-\bm X{\bm\beta}^t);s\lambda)
\label{ista}
\end{align}

\paragraph{FISTA (Nesterov-accelarated ISTA)}
\begin{equation}
\begin{split}
{\bm\beta}^{t+1}&=\eta_\gamma({\bm M}^{t}+s \bm X^\top(\bm y-\bm X{\bm M}^t);s\lambda)
\\
d_{t+1}&=(1+\sqrt{1+4d_t^2})/2
\\
\bm M^{t+1}&=\bm\beta^t+\frac{d_t-1}{d_{t+1}}(\bm\beta^t-\bm\beta^{t-1})
\end{split}
\label{fista}
\end{equation}
where $d_1 = 1, s \in (0, 1/\|\mathbf{X} \|_F)$.
\\\\
\begin{algorithm}[H]
\centering 
\caption{Blockwise descent \cite{simon2013sparse}}
\label{alg:block}
\begin{algorithmic}
\STATE{\textbf{Input}:
Data matrix $\mathbf{X}$, response $\mathbf{y}$, group information $\bm g$, number of iterations $T$, step size $s$ } \\
 \FOR{$t$ in $1:T$}
    \FOR{ $l$ in $1: L$ }
    \IF{$\|\eta_{\text{soft}}(\mathbf{X}_l \mathbf{r}_{(-l)}, \gamma \lambda ) \|_2 \leq (1 - \gamma) \lambda $}
    \STATE{$\bm{\hat{\beta}}_{l} \leftarrow U_\gamma(\bm{\hat{\beta}}_{l}- s \bm X_l^\top(\bm X_l\bm{\hat{\beta}}_{l}-\mathbf{r}_{(-l)}), s)$}
    \ELSE
    \STATE{$\bm{\hat{\beta}}_{l}\leftarrow \bm 0$}
    \ENDIF
    \ENDFOR
    \ENDFOR
\end{algorithmic}
\end{algorithm}
Here the update rule, for group $l$, is
\begin{align*}
U_\gamma(\bm \theta,s) =& \left(1 -\frac{s(1-\gamma)\lambda\sqrt{p_l}}{||\eta_{\text{soft}}(\bm \theta,s\gamma \lambda )||_2}\right)_+
\cdot\eta_{\text{soft}}(\bm \theta,s\gamma \lambda ). 
\end{align*}

For ISTA, if we denote $\bm v_l:=\bm {\hat{\beta}}_l - s \mathbf{X}_l^{\top} (\mathbf{X}_l\bm{\hat{\beta}}_l - \mathbf{y})$, then after some algebra we have:
\begin{algorithm}[!htb]
\centering
 \caption{ISTA (Proximal Gradient Descent)}
 \label{alg:ista}
\begin{algorithmic}
\STATE{\textbf{Input}: Data matrix $\mathbf{X}$, response $y$, group information $\bm g$, number of iterations $T$, step size $s\leq 1/\|\mathbf{X}^\top \mathbf{X}\|_F$}
 \FOR{$t$ in $1:T$}
     \FOR{ $l$ in $1: L$ }
     \IF{$\mathcal{C}(\bm v_l,U_\gamma(\bm v_l,s))<\mathcal{C}(\bm v_l,\bm 0)$}
 \STATE{$\bm {\hat\beta}_l = U_\gamma(\bm {\hat\beta}_l - s \mathbf{X}_l^\top (\mathbf{X}_l\bm{\hat\beta}_l - \mathbf{y}), s)$ }
 \ELSE
 \STATE{$\bm{\hat\beta}_l\leftarrow\bm 0$}
 \ENDIF
 \ENDFOR
 \ENDFOR
\end{algorithmic}
\end{algorithm}

We note that both approaches are examples of majorization-minimization algorithms. But we highlight that the proximal approach updates the estimation groupwise independently in the sense that $\bm r_{(-l)}$ is not required, which depends on the estimates in all other groups. We remark that computing $\bm r_{(-l)}$ for each $l$ has complexity $O(np)$ and is computationally more expensive than directly calling $\bm y$. Furthermore, if other estimates $\bm{\hat\beta}_{k\neq l}$ are not accurate enough, then the update of $\bm{\hat\beta}_l$ suffers and is therefore not efficient. In Figure \ref{fig:times}, we plot the the convergence of the accelerated blockwise descent in R-package \texttt{SGL} against time, together with the convergence of proximal methods. In terms of the loss function $\mathcal{C}$ and the error $|\mathcal{C}-\min_{\beta}\mathcal{C}|$, even un-accelerated proximal gradient descent converges much faster in time than the accelerated blockwise descent. Especially, the complexity per iteration is much larger for the blockwise descent. In addition, the independence of updates between groups allows the algorithm to further speedup via the parallel computing within each group.

As for the rigorous convergence analysis, we claim from the analysis of general proximal gradient descent in \cite{parikh2014proximal}, that ISTA and FISTA is guaranteed to converge on any data matrix $\bm X$ and the convergence rates are $O(1/t)$ for ISTA and $O(1/t^2)$ for FISTA. 

We can easily extend the proximal gradient to other models by minimizing
\begin{equation} \label{extended objective}
    \min_{\beta \in \mathbb{R}^p} L_{\bm X,\bm y}(\bm\beta)  + (1 - \gamma)\lambda \sum_{l=1}^L \sqrt{p_l} ||\bm\beta_l||_2 +  \gamma \lambda ||\bm\beta||_1
\end{equation}
and guarantees the convergence as long as $\nabla L$ is Lipschitz continuous. For example, we can consider the logistic regression $L(\bm\beta)=\frac{1}{n}\sum_{i=1}^n [\log\left(1+\exp(x_i^\top\bm\beta)\right)+y_i x_i^\top\bm\beta]$. The generalized proximal gradient descent is
\begin{align*}
{\bm\beta}^{t+1}=\eta_\gamma({\bm\beta}^{t}-s_t \nabla L(\bm\beta^t);s_t\lambda).
\end{align*}

\section{Proof of Main Results} \label{Appendix C}
We first prove Theorem \ref{theorem1}, assuming that Theorem \ref{theorem2} and Lemma \ref{berthier} hold.

\begin{proof}[Proof of Theorem \ref{theorem1}]
The proof is similar to \cite{bu2019algorithmic} Theorem 3. 
By Definition \ref{lipschitz} of the uniformly pseudo-Lipschitz function of order $k$ and by the Triangle Inequality, we obtain that for any $p$ and $t$, 
\begin{align*}
 &\quad    |\varphi_p (\bm \beta^t, \bm \beta ) - \varphi_p ( \hat{\bm \beta}, \bm \beta )  |/ \|\bm \beta^t - \hat{ \bm \beta}  \|\\
 \leq& \frac{L}{\sqrt{2p}} \left(1 + \left(\frac{\| (\bm \beta^t, \bm \beta ) \|}{\sqrt{2p}}\right)^{k-1} + \left(\frac{\| ( \hat{ \bm \beta} , \bm \beta ) \|}{\sqrt{2p}}\right)^{k-1} \right)  \\
 \leq &\frac{L}{\sqrt{2p}}\left(1 + \Big(\frac{\| \bm \beta^t \|}{\sqrt{2p}}\Big)^{k-1} + \Big(\frac{\|  \hat{ \bm \beta}  \|}{\sqrt{2p}}\Big)^{k-1} + \Big(\frac{\|\bm \beta \|}{\sqrt{2p}}\Big)^{k-1} \right) \\ 
\end{align*}
Taking limits first with respect to $p$ and then with respect to $t$, we notice that by Theorem \ref{theorem2}, $\frac{1}{p} \|\bm \beta^t - \hat{\bm \beta} \|$ converges to zero and thus the left hand side goes to infinity, unless the numerator converges to zero as well. We show that that this is the case.

For the term on the right hand side, we claim it is finite as 
\begin{equation} \label{B.2}
  \lim_t \lim_p \Big(\frac{\| \bm \beta^t \|}{\sqrt{2p}}\Big)^{k-1},
    \lim_t \lim_p \Big(\frac{\| \hat{\bm \beta} \|}{\sqrt{2p}}\Big)^{k-1}, 
    \lim_t \lim_p \Big(\frac{\| \bm \beta \|}{\sqrt{2p}}\Big)^{k-1}
\end{equation}
are all finite. Consequently we have:
\begin{align*}
    \lim_p \varphi_p (\hat{\bm \beta}, \bm \beta) = \lim_t \lim_p \varphi_p (\bm \beta^t, \bm \beta).
\end{align*}
The proof is complete once we see from Lemma \ref{berthier} that
\begin{align*}
    \lim_t \lim_p \varphi_p (\bm \beta^t, \bm \beta) = \lim_t \lim_p \mathbb{E}[\varphi_p (\eta_{\gamma}(\bm \beta + \tau_t \bm Z; \alpha \tau_t),  \bm \beta )].
\end{align*}

Finally, we show that our claim about \eqref{B.2} holds. Consider the first term in \eqref{B.2}: $\bm \beta^{t+1} = \eta_\gamma (\mathbf{X}^\top \mathbf{z} + \bm \beta^t, \alpha\tau_t)$.

Applying Lemma \ref{berthier} on $\varphi_p (\bm \beta^t + \mathbf{X}^\top \mathbf{z}^t, \bm \beta) = \frac{1}{p} \| \eta_\gamma (\bm \beta^t+\mathbf{X}^\top \mathbf{z}, \alpha\tau_t) \|$ gives
\begin{align}\label{liit}
\lim_p \frac{1}{p} \| \bm \beta^t \|^2 = \lim_p \frac{1}{p} \mathbb{E}\|\eta_\gamma (\bm \beta + \tau_t \mathbf{Z}, \alpha\tau_t)\|^2
\end{align}
for $\mathbf{Z} \sim \mathcal{N}(0, \mathcal{I}_p)$. By property (\textbf{S1}), we also have
\begin{align*}
    \mathbb{E}\|\eta_\gamma (\bm \beta + \tau_t \mathbf{Z}, \alpha\tau_t)\|^2 \leq \mathbb{E} \|\bm \beta + \tau_t \mathbf{Z} \|^2\leq 2 \|\bm \beta \|^2 + 2p\tau_t^2.
\end{align*}
Combining the above inequality and \eqref{liit}, and using Assumption (\textbf{A2}) and Property (\textbf{S2}), we get
\begin{align*}
     \lim_t \lim_p \frac{\| \bm \beta^t \|^2}{p} \leq 2 \mathbb{E}\Pi^2 + 2 \tau_*^2
\end{align*}
which is equivalent to
\begin{align*}
     \lim_t \lim_p \left(\frac{\| \bm \beta^t \|}{\sqrt{2p}}\right)^{k-1} \leq (\mathbb{E}\Pi^2 + \tau_*^2)^{\frac{k-1}{2}}.
\end{align*}

Then it is not hard to see that the second term in \eqref{B.2} is finite from Theorem \ref{theorem1} and the bound on the first term. Lastly, the third term in \eqref{B.2} is finite by Property \textbf{(S2)}.
\end{proof}

We sketch the proof of Theorem \ref{theorem2} here:

\begin{proof}[Proof of Theorem \ref{theorem2}]
Our proof relies on the technical result in \cite{bayati2011lasso} Lemma 3.1 to guarantee that, although the SGL cost function $\mathcal{C}_{\bm X,\bm y}(\cdot)$ is not necessarily strictly convex, we do not encounter the case that the cost of $\mathcal{C}_{\bm X,\bm y}(\bm\beta^t)$ is close to $\mathcal{C}_{\bm X,\bm y}(\bm{\hat\beta})$ yet $\bm\beta^t$ is far from $\bm{\hat\beta}$.

In the LASSO case, nice properties exist when one only considers columns of $\bm X$ corresponding to the non-zero elements in $\bm\beta^t$. Our proof designs specific subgradients of $\|\cdot\|_2$ to relate to the LASSO case in the sense that we also consider the selected elements in $\bm\beta^t$. We see an analogy between the SGL proximal operator and soft-thresholding in \eqref{eta}. Specifically, the SGL proximal operator can be viewed as a scaled or shrinked soft-thresholding. In the LASSO case, the support set has size no larger than $n$, i.e. the sum of indicating functions of being non-zero $\leq n$; in the SLOPE case, the support of equivalence classes has size no larger than $n$, which means the weighted sum of indicator functions is $\leq n$ with each element in the same equivalence class sharing the same weight. We claim that in SGL, similarly to SLOPE, the weighted sum of indicator functions is $\leq n$, though weights depend on the scaling between the SGL proximal operator and the soft-thresholding.
\end{proof}

In order to apply Lemma \ref{berthier}, we need to check the proximal operator of SGL satisfies the following properties, and the verification of these properties is shown in Appendix \ref{Appendix C}:

\begin{itemize}
    \item[\textbf{(S1)}] For each $t$, the proximal operators are uniformly Lipschitz (i.e. uniformly pseudo-
    Lipschitz of order $k = 1$).
    \item[\textbf{(S2)}] For any $s,t$ with $(\mathbf{Z}, \mathbf{Z}')$ a pair of length $p$ vectors such that $(Z_i, Z_i')$ are i.i.d $\mathcal{N}(\mathbf{0},\mathbf{\Sigma})$ for $i \in \{1,2,\cdots, p \}$ where $\mathbf{\Sigma}$ is any $2 \times 2$ covariance matrix, the following limits exist and are finite:
    \begin{equation}
    \begin{split}
    &    \lim_{p \rightarrow \infty} \frac{1}{p} \| \bm \beta \|_2 \\
     &  \lim_{p \rightarrow \infty} \frac{1}{p} \mathbb{E} \left( \bm \beta^T  \eta_\gamma ( \bm \beta + \mathbf{Z}, \alpha \tau_t) \right) \\
     & \lim_{p \rightarrow \infty}  \frac{1}{p} \mathbb{E} \left( \eta_\gamma ( \bm \beta + \mathbf{Z}', \alpha \tau_s)^\top \eta_\gamma ( \bm \beta + \mathbf{Z}, \alpha \tau_t)   \right)
    \end{split}
    \end{equation}
\end{itemize}

Now we check that our proximal operator satisfies the assumptions in \cite{berthier2017state}.

\begin{proof}[Verifying Properties (\textbf{S1}) and (\textbf{S2})]\label{proof 38}
Property (\textbf{S1}) follows from
\begin{align*}
    \|\eta_\gamma (\bm v_1, \alpha \tau_t) - \eta_\gamma (\bm v_2, \alpha \tau_t)   \|_2 \leq \| \eta_{\text{soft}} (\bm v_1, \gamma \alpha \tau_t) - \eta_{\text{soft}} (\bm v_2, \gamma \alpha \tau_t)    \|_2 \leq \| \bm v_1 - \bm v_2  \|_2
\end{align*}
as the SGL proximal operator can be viewed as a shrinked soft-thresholding. Hence, the proximal operators are Lipschitz continuous with Lipschitz constant one. 

To show (\textbf{S2}) holds, we restate Property (\textbf{S2}) here:
for any $s,t$ with $(\mathbf{Z}, \mathbf{Z}')$ a pair of length $p$ vectors such that $(Z_i, Z_i')$ are i.i.d. $\mathcal{N}(\mathbf{0},\mathbf{\Sigma})$ for $i \in \{1,2,\cdots, p \}$ where $\mathbf{\Sigma}$ is any $2 \times 2$ covariance matrix, the following limits exist and are finite:
     \begin{align}
    &    \lim_{p \rightarrow \infty} \frac{1}{p} \| \bm \beta \|_2 \label{D1} \\
     &  \lim_{p \rightarrow \infty} \frac{1}{p} \mathbb{E} \left( \bm \beta^T  \eta_\gamma ( \bm \beta + \mathbf{Z}, \alpha \tau_t) \right) \label{D2} \\
     & \lim_{p \rightarrow \infty}  \frac{1}{p} \mathbb{E} \left( \eta_\gamma ( \bm \beta + \mathbf{Z}', \alpha \tau_s)^\top \eta_\gamma ( \bm \beta + \mathbf{Z}, \alpha \tau_t)   \right). \label{D3}
    \end{align}

The first limit \eqref{D1} can be easily verified by the strong Law of Large Number and Assumption (\textbf{A2}). We need to apply the following two lemmas to prove the other two limits. The first lemma will produce the dominated convergence result, that we use many times to exchange the limit and the expectation.

\begin{lemma}[Doob's $L^1$ maximal inequality, \cite{JLDoob}]
\label{C1}
    Let $X_1, \cdots, X_p$ be a sequence of non-negative i.i.d. random variables such that $\mathbb{E} \{ X_1 \max(0, \log X_1)\} < \infty$. Then
    \begin{align*}
        \mathbb{E}\left\{ \sup_{p \geq 1} \frac{X_1 + \cdots + X_p}{p}   \right\} 
        \leq &\frac{e}{e-1} \left(1 + \mathbb{E}\{X_1 \max(0, \log X_1)) \} \right)
    \end{align*}
    \end{lemma}
\begin{proof}[Proof of \Cref{C1}]
    Let $M_p = \frac{1}{p} (X_1 + \cdots + X_p)$. Then the sequence $\{M_p \}$ is a submartingale and hence by Doob's maximal inequality, 
    \begin{align*}
        \mathbb{E} \left\{\sup_{p' \geq p \geq 1} M_p  \right\} 
        \leq &\frac{e}{e-1} \left(1 + \mathbb{E} \{ M_{p'} \max(0, \log M_{p'})\} \right)
    \end{align*}
    Notice that the mapping $x \rightarrow x \max(0, \log x)$ is convex and hence 
    \begin{align*}
        \mathbb{E} \{ M_{p'} \max(0, \log M_{p'}) \} \leq \mathbb{E} \{ X_{1} \max(0, \log X_{1}) \}
    \end{align*}
     The result follows by $\sup_{p' \geq p \geq 1} M_p \uparrow \sup_{p \geq 1} M_p$ as $p' \rightarrow \infty$ and by Fatou's lemma. 
    \end{proof}

The next lemma we need uses the idea from Proposition 1 in \cite{hu2019asymptotics}, that the non-separable proximal operator is indeed asymptotically separable.

\begin{lemma}\label{C2}
For a penalty sequence $\{\lambda(p) \}$, having empirical distribution that weakly converge to a distribution $\Lambda$, there exists a limiting scalar function $h$ such that as $p \rightarrow \infty$,
    \begin{equation} \label{E1}
        \frac{1}{p} \|\eta_\gamma (\bm v (p), \lambda) - h(\bm v(p); \Lambda)  \|^2 \rightarrow 0
    \end{equation}
    where $h$ applies $h(\cdot;\Lambda)$ coordinate-wise to $\bm v(p)$ and $h$ is Lipschitz with constant $1$.
\end{lemma}

\begin{proof}[Proof of Lemma \ref{C2}]
Recall that the SGL proximal operator in \eqref{eta} takes the form
\begin{align*}
\eta_{\gamma}(\bm s, \lambda)^{(j)} =  \eta_{\text{soft}}(s^{(j)}, \gamma  \lambda)\left(1 - \frac{(1 - \gamma)\lambda \sqrt{p_{l_j}}}{ \| \eta_{\text{soft}}(\bm s_{l_j}, \gamma \lambda) \|_2 }\right).
\end{align*}
We highlight that the soft-thresholding is a separable operator and the scalar term converges to $1 - \frac{(1 - \gamma)\lambda}{ \sqrt{\mathbb{E} \eta_{\text{soft}}(S_j, \gamma \lambda)^2}}$, where $S_j$ is the asymptotic distribution of $\bm s_{l_j}$.
\end{proof}

We apply Lemma \ref{C2} to verify limits \eqref{D2}\eqref{D3}.

First, consider limit \eqref{D2}, by Cauchy-Schwarz, we have:
\begin{equation}
     \frac{1}{p} | \bm \beta^\top \eta_\gamma (\bm \beta + \mathbf{Z}, \alpha \tau_t) - \bm \beta^\top h^t (\bm \beta + \mathbf{Z})  | \leq \frac{\|\bm \beta \|_2}{\sqrt{p}} \frac{\|\eta_\gamma (\bm \beta + \mathbf{Z}, \alpha \tau_t) - h^t (\bm \beta + \mathbf{Z}) \|_2}{ \sqrt{p}}
\end{equation}
    which goes to $0$ by (\ref{D1}) and (\ref{E1}).  This means that the assumptions of Lemma~\ref{C1} are satisfied and it implies that
\begin{align*}
    |\mathbb{E}_{\bm Z} \{ \bm \beta^\top \eta_{\gamma}( \bm\beta + \bm Z, \alpha \tau_t)\} -  \mathbb{E}_{\bm Z} \{\bm \beta^\top h^t( \bm\beta + \bm Z) \}| /  p \rightarrow 0
\end{align*}
as $p \rightarrow \infty.$
Therefore,
\begin{align*}
 \lim_{p\to\infty} \mathbb{E}_{\bm Z} \{\bm \beta^{\top} \eta_{\gamma}( \bm\beta + \bm Z, \alpha \tau_t)\}/p
=\lim_{p\to\infty} \sum_{i=1}^{p}\beta_{0,i}  \mathbb{E}_{Z}\{h^t (\beta_{0,i}+Z_i)\}/p= \mathbb{E}[\Pi^\top h^t(\Pi+Z)],
\end{align*}
where $\Pi, Z$ are univariate.
By the Cauchy-Schwarz inequality, we need $\mathbb{E}[\Pi^2]<\infty$ and $\mathbb{E}[h^t(\Pi+Z)^2]<\infty$ to bound $\mathbb{E}[\Pi h^t(\Pi+Z)]$. Since $\mathbb{E}[\Pi^2] := \sigma_{\bm\beta}^2 <\infty$ is assumed, we focus to show $\mathbb{E}[h^t(\Pi+Z)^2]<\infty$. Indeed, $h^t(\cdot)$ is Lipschitz with constant 1 and therefore $\mathbb{E}[h^t(\Pi+Z)^2]<\mathbb{E}[(\Pi+Z)^2] \leq \mathbb{E}[\Pi^2] + \mathbb{E}[Z^2] = \sigma_{\bm\beta}^2  + \Sigma_{11} < \infty.$

Finally consider the last limit \eqref{D3} similarly. We appeal to Lemma~\ref{C1} which requires that
\begin{equation}
\frac{1}{p}|\eta_{\gamma}(\bm\beta + \mathbf{Z}', \alpha \tau_s)^\top \eta_{\gamma}(\bm\beta + \mathbf{Z}, \alpha \tau_t)  - h^s(\bm\beta + \mathbf{Z}')^\top h^t(\bm\beta + \mathbf{Z})| 
\label{eq:P2_res2}
\end{equation}
goes to $0$ as $p \rightarrow \infty$.

Now we rigorously prove \eqref{eq:P2_res2}.  
By repeated applications of the Cauchy-Schwarz inequality, we find
\begin{align*}
&\lim_{p \rightarrow \infty} |\eta_{\gamma}(\bm\beta + \mathbf{Z}', \alpha \tau_s)^\top \eta_{\gamma}(\bm\beta + \mathbf{Z}, \alpha \tau_t) - h^s(\bm\beta + \mathbf{Z}')^\top h^t(\bm\beta + \mathbf{Z})| \\
&\leq \lim_{p \rightarrow \infty} \| h^s(\bm\beta + \mathbf{Z}')\|  \| \eta_{\gamma}(\bm\beta + \mathbf{Z}, \alpha \tau_t)   -  h^t(\bm\beta + \mathbf{Z})\|  \\
&+  \lim_{p \rightarrow \infty} \| h^t(\bm\beta + \mathbf{Z}) \| \|  \eta_{\gamma}(\bm\beta + \mathbf{Z}',\alpha \tau_s)   -  h^s(\bm\beta + \mathbf{Z}') \| \\
& + \lim_{p \rightarrow \infty} \| \eta_{\gamma}(\bm\beta + \mathbf{Z}',\alpha\tau_s) - h^s(\bm\beta + \mathbf{Z}') \|   \| \eta_{\gamma}(\bm\beta + \mathbf{Z},\alpha \tau_t) - h^t(\bm\beta + \mathbf{Z}) \|.
\end{align*}
We claim that \eqref{eq:P2_res2} converges to 0 as the result of both
\begin{align*}
  &\|   \eta_{\gamma}(\bm\beta + \mathbf{Z}',\alpha \tau_s) - h^s(\bm\beta + \mathbf{Z}') \| /\sqrt{p} \rightarrow 0, 
\\
 &\|    \eta_{\gamma}(\bm\beta + \mathbf{Z},\alpha \tau_t)   -  h^t(\bm\beta + \mathbf{Z}) \| /\sqrt{p} \rightarrow 0
\end{align*}
and
\begin{align*}
\|h^s(\bm\beta+\bm Z'\|/p\to 0, \|h^t(\bm\beta+\bm Z\|/p\to 0.
\end{align*}

While the first pair of convergence results follow from Lemma ~\ref{C2}, the second pair follow from the fact that $h^s(\cdot)$ and $h^t(\cdot)$ are separable and from the Law of Large Numbers,
\begin{equation*}
\begin{split}
&\lim_{p} \norm{h^s(\bm\beta + \mathbf{Z}')}_2^2/p = \lim_{p}  \sum_{i=1}^{p} [h^s(\beta_{i} + Z'_i)]^2/p 
\\
= &\mathbb{E}[(h^s(\Pi + Z'))^2]   \leq \mathbb{E}[(\Pi + Z')^2]
\leq \sigma_{\bm\beta}^2 + \Sigma_{22} < \infty, \\
& \lim_{p} \norm{h^t(\bm\beta + \mathbf{Z})}^2/p = \lim_{p}\sum_{i=1}^{p} [h^t(\beta_{i} + Z_i)]^2/p 
\\
= &\mathbb{E}[(h^t(\Pi + Z))^2]   \leq \mathbb{E}[(\Pi + Z)^2] \leq \sigma_{\bm\beta}^2 + \Sigma_{11}< \infty.
\end{split}
\end{equation*}

Next, \eqref{eq:P2_res2} allows Lemma~\ref{C1} to imply, 
\begin{align*}
 & \lim_{p\to\infty} \mathbb{E}_{\bm Z, \bm Z'}\{\eta_{\gamma}(\bm\beta + \mathbf{Z}', \alpha \tau_s)^\top \eta_{\gamma}(\bm\beta + \mathbf{Z},\alpha \tau_t)\}/p \\ 
 &=\lim_{p\to\infty} \sum_{i=1}^{p} \mathbb{E}_{\bm Z, \bm Z'}\{h^s(\beta_{i} + Z'_i) h^t(\beta_{i} + Z_i) \}/p  \\
 &= \mathbb{E}[h^s(\Pi+Z')^\top h^t(\Pi+Z)],
\end{align*}
where $\Pi, Z',$ and $Z$ are univariate. By the Cauchy-Schwarz inequality and the Lipschitz property, we find
\begin{align*}
&\Big(\mathbb{E}[h^s(\Pi+Z')^\top h^t(\Pi+Z)]\Big)^2 
\leq  \mathbb{E}[(h^s(\Pi+Z'))^2] \mathbb{E}[(h^t(\Pi+Z))^2] 
\leq \mathbb{E}[(\Pi+Z')^2] \mathbb{E}[(\Pi+Z)^2] \\
= & ( \mathbb{E}[\Pi^2]+ \mathbb{E}[Z'^2] )(\mathbb{E}[\Pi^2]+ \mathbb{E}[Z^2] ) 
= (\sigma_{\bm\beta}^2+\Sigma_{22})(\sigma_{\bm\beta}^2+\Sigma_{11}) <\infty
\end{align*}
which leads to the boundedness of \eqref{D3}.
\end{proof}

\section{Analysis of State Evolution} \label{Appendix D}
\label{app:se}

\begin{proof}[Proof of Proposition \ref{prop:F}]
The proof structure is significantly extended from Proposition 1.3 in \cite{bayati2011lasso}. For the sake of brevity, we consider the case where only one group exists and claim the proof can be easily generalized to the multiple group case since each group evolves independently. 

The main focus is to show
that $\mathbf{F}_\gamma (\tau^2, \alpha \tau):= \sigma_{\bm w}^2 + \frac{1}{\delta p} \mathbb{E}\|\eta_{\gamma}(\mathbf{\Pi} + \tau \bm Z, \alpha \tau) - \mathbf{\Pi} \|_2^2$ is concave in $\tau^2$. Equivalently, we show that $\mathbb{E} \|\eta_{\gamma}(\bm \Pi + \tau \bm Z,  \alpha \tau) - \bm \Pi \|^2/p$ is concave.

Expanding the SGL proximal operator, we have
\begin{align*}
   &\qquad\mathbb{E} \|\eta_{\gamma}(\bm \Pi + \tau \bm Z,  \alpha \tau) - \bm \Pi \|^2  \\
    &= \mathbb{E}\left\|\bm\eta_{\text{soft}} \cdot\left(1 - \frac{(1 - \gamma)\alpha\tau \sqrt{p_l}}{ \| \bm\eta_{\text{soft}}\|_2 }\right) - \bm \Pi \right\|^2
\end{align*}
with $\bm\eta_{\text{soft}}\in\mathbb{R}^{p_l}$ denoting $\eta_{\text{soft}}(\bm \Pi + \tau \bm Z, \gamma  \alpha \tau)$.

By the Law of Large Numbers, almost surely,
\begin{align*}
    \frac{\sqrt{p_l}}{||\bm\eta_{\text{soft}}||_2}
    \to \frac{1}{\sqrt{ \mathbb{E}\eta_{\text{soft}}\left(\Pi + \tau Z, \gamma \alpha \tau \right)^2    }}.
\end{align*}

Since the soft-thresholding is separable, it suffices to show that
$$\mathbb{E}\left(\eta_{\text{soft}} \cdot\left(1 - \frac{(1 - \gamma)\alpha\tau }{ \sqrt{\mathbb{E}\eta_{\text{soft}}^2}}\right) -\Pi\right)^2$$
is concave. Simple binary expansion gives
\begin{align*}
&\quad\mathbb{E}\left(\eta_{\text{soft}} \cdot\left(1 - \frac{(1 - \gamma)\alpha\tau }{ \sqrt{\mathbb{E}\eta_{\text{soft}}^2}}\right) -\Pi\right)^2  =\mathbb{E}\Bigg[(\eta_{\text{soft}} - \Pi  )^2- \frac{2(1-\gamma)\alpha \tau \eta_{\text{soft}}}{\sqrt{\mathbb{E}\eta_{\text{soft}}^2} }(\eta_{\text{soft}}-\Pi) + \frac{(1-\gamma)^2\alpha^2 \tau^2 \bm\eta_{\text{soft}}^2}{\mathbb{E}\eta_{\text{soft}}^2}\Bigg]
  \\
  & = \mathbb{E} (\eta_{\text{soft}}- \Pi  )^2+2(1 - \gamma) \alpha\tau \Pi \frac{\mathbb{E}\eta_{\text{soft}}}{\sqrt{\mathbb{E}\eta_{\text{soft}}^2}} - 2(1-\gamma)\alpha\tau \sqrt{\mathbb{E}\eta_{\text{soft}}^2}    + (1 - \gamma)^2 \alpha^2 \tau^2 .
\end{align*}

It has been shown in \cite{bayati2011lasso} that $\mathbb{E}\left(\eta_{\text{soft}}(\Pi + \tau Z, \gamma \alpha \tau) - \Pi \right)^2$
is concave. Additionally, the last term $(1-\gamma)^2\alpha^2\tau^2$ is linear (hence concave) in terms of $\tau^2$. Therefore, it remains to prove that
\begin{align*}
    H_\gamma (\tau^2):=\tau\Pi\frac{\mathbb{E}\eta_{\text{soft}}  }{\sqrt{\mathbb{E}\eta_{\text{soft}}^2}  } - \tau \sqrt{\mathbb{E}\eta_{\text{soft}}^2}
\end{align*}
is concave in $\tau^2$ for any $\Pi$.

We exhibit the explicit formulae for the expectation terms by viewing the soft-thresholding with Gaussian input as a truncated normal distribution. With some calculus, the moments of the truncated normal distribution give
\begin{align*}
  & \qquad \mathbb{E}\left( \eta_{\text{soft}}(\Pi + \tau Z, \gamma \alpha \tau)  \right)^2  
  \\
  =&\tau^2\Bigg[\left(1 + \left(\frac{-\Pi + \gamma \alpha \tau}{\tau}\right)^2\right) \Phi\left(\frac{\Pi - \gamma \alpha \tau}{\tau}\right)  - \left(\frac{-\Pi + \gamma \alpha \tau}{\tau}\right)\phi\left(\frac{-\Pi + \gamma \alpha \tau}{\tau}\right) \\
  &\quad\quad+ \left(1 + \left(\frac{\Pi + \gamma \alpha \tau}{\tau}\right)^2\right) \Phi\left(\frac{-\Pi - \gamma \alpha \tau}{\tau}\right) - \left(\frac{\Pi + \gamma \alpha \tau}{\tau}\right) \phi\left(\frac{\Pi + \gamma \alpha \tau}{\tau}\right)\Bigg]
\end{align*}
and
\begin{align*}
    &\qquad \mathbb{E}\eta_{\text{soft}}(\Pi + \tau Z, \gamma \alpha \tau) \\
    & = \tau\Bigg[\left(\frac{\Pi - \gamma \alpha \tau}{\tau}\right)\Phi\left(\frac{\Pi - \gamma \alpha \tau}{\tau}\right) + \phi\left(\frac{-\Pi + \gamma \alpha \tau}{\tau}\right) 
    \\
    &-  \left(\frac{-\Pi - \gamma \alpha \tau}{\tau}\right)\Phi\left(\frac{-\Pi - \gamma \alpha \tau}{\tau}\right) - \phi\left(\frac{-\Pi - \gamma \alpha \tau}{\tau}\right)    \Bigg]
\end{align*}

We discuss two cases conditioned on $\Pi$, corresponding to whether the signal is null or not. When $\Pi = 0$, we have:
\begin{align*}
    H_\gamma (\tau^2) = - \tau \sqrt{\mathbb{E}\eta_{\text{soft}}^2(\tau Z, \gamma \alpha \tau) }  
    = -\tau^2  [ 2(1 + \gamma^2 \alpha^2) \Phi(-\gamma \alpha) - 2\gamma \alpha\phi(\gamma \alpha)]^{1/2}
\end{align*}
which is linear in $\tau^2$ and thus concave.

When $\Pi \neq 0 $, we consider the one-side case for simplicity by assuming that the signal is non-negative and claim that the other side holds by symmetry. In particular, the corresponding one-side soft-thresholding is $\Tilde{\eta}_{\text{soft}}(x;b)=\max(x-b, 0)$. 

Denote $\Tilde{H}_\gamma$ as the one side case of $H_\gamma$, let $c = \frac{-\Pi + \gamma \alpha \tau}{\tau}$ so that equivalently $\tau=\frac{\Pi}{ \gamma \alpha - c}$, we have
\begin{align*}
    \Tilde{H}_\gamma (c) = \frac{1}{2} \frac{\Pi^2}{ \gamma \alpha - c} \frac{T'(c)}{\sqrt{T(c)}} - \left(\frac{\Pi}{ \gamma \alpha - c}\right)^2 \sqrt{T(c)}
\end{align*}
where $T(z)=(1+z^2)\Phi(-z)-z\phi(z)$ is as defined in Proposition \ref{prop:F} and it is easy to see $\tau^2 T(c)=\mathbb{E}\left(\Tilde{\eta}_{\text{soft}}(\Pi+\tau Z, \gamma\alpha\tau)\right)^2$ by looking at the second moment of the truncated normal distribution.

Then by the chain rule, we obtain
\begin{align*}
    \frac{d \Tilde{H}_\gamma}{d \tau^2} &= \frac{d\Tilde{H}_\gamma}{dc}\cdot \frac{dc}{d \tau}/\frac{d\tau^2}{d\tau} = \Bigg[\frac{-\Pi^2}{(\gamma \alpha)^2 - c} \left( \frac{T'(c)}{\sqrt{T(c)}} + \frac{1}{2} \frac{1}{\sqrt{T(c)}} \right) + \frac{\Pi^2}{ \gamma \alpha - c}\left(\frac{T''(c)}{\sqrt{T(c)}} - \frac{T'(c)}{2(T(c))^{3/2}} \right) \\
    &+ \frac{\Pi^2}{2(\gamma \alpha - c)^3}\sqrt{T(c)} \Bigg] \frac{\Pi}{2 \tau^3} = - \tau^2 \sqrt{T(\frac{- \Pi + \gamma \alpha \tau}{ \tau})}
\end{align*}
which decreases as $\tau^2$ increases, leading to
\begin{align*}
    \frac{d^2\Tilde{H}_\gamma}{d (\tau^2)^2} < 0 
\end{align*}
Hence $\Tilde{H}_\gamma$, as well as $H_\gamma$, is concave. Therefore, $\mathbf{F}_{\gamma}$ is concave in terms of $\tau^2$ by the previous argument. 

Employing the concavity, we can derive the following convergence results. Notice that for all $\alpha \geq 0$,
\begin{align*}
\lim_{\tau^2 \rightarrow \infty} \frac{d \mathbf{F}_{\gamma}}{d\tau^2} 
= \frac{1}{\delta} \Big( 2T(\gamma \alpha) - 2(1- \gamma)\alpha \sqrt{2T(\gamma \alpha)} + (1-\gamma)^2 \alpha^2 \Big)
=\frac{1}{\delta} \Big( \sqrt{2T(\gamma \alpha)} - (1- \gamma)\alpha \Big)^2> 0
\end{align*}
where the non-negativity of $T(\gamma\alpha)$ is guaranteed by observing that $T'(z)=2z\Phi(-z)-2\phi(z)<0$ and $T(\infty)=0$. 

From the fact that $\frac{d \mathbf{F}_{\gamma}}{d\tau^2}>0$ for large $\tau$ and that $\mathbf{F}_\gamma$ is concave, it is clear that $\mathbf{F}_\gamma$ is increasing everywhere.

Clearly $\mathbf{F}_\gamma(\tau^2,\alpha\tau)=\sigma_w^2>\tau^2=0$ when $\tau=0$. Also for $ \alpha\in\mathcal{A}$, we have $\lim_{\tau^2 \rightarrow \infty} \frac{d \mathbf{F}_{\gamma}}{d\tau^2}<1$ by the definition of $\mathcal{A}$. In other words, 
$\mathbf{F}_\gamma (\tau^2, \alpha \tau) > \tau^2$ for $\tau^2$ small enough and $\mathbf{F}_\gamma (\tau^2, \alpha \tau) < \tau^2$ for $\tau^2$ large enough. Hence the state evolution has at least one solution. Furthermore, since $\mathbf{F}_{\gamma}$ is concave and increasing in terms of $\tau^2$, the solution must be unique and $\tau_t^2 \rightarrow \tau_*^2$.

\end{proof}

\section{Analysis of Calibration} \label{Appendix E}
\label{app:cali}

\begin{proof}[Proof of Lemma \ref{lemma:3.2}]
Our proof is within group level, and since all groups are disjoint, this can be extended to whole level. Let $\hat{\bm{\beta}}_l  = \eta_\gamma(\mathbf{r},\theta_*)_l$, denote $w_l =  \frac{1}{\delta} \langle \eta_\gamma'((\mathbf{X}^*\hat{\mathbf{z}} + \hat{\bm{\beta}})_l, \theta_*) \rangle, w =  \frac{1}{\delta} \langle \eta_\gamma'(\mathbf{X}^*\hat{\mathbf{z}} + \hat{\bm{\beta}}, \theta_*) \rangle$, then by the definition of $\eta_\gamma$, for simplicity assume $\hat{\bm{\beta}}_\ell\neq \mathbf{0}$, we have:
\begin{align*}
    \eta_{\text{soft}}(\mathbf{r}, \gamma  \theta_*)_l \left(1 - \frac{(1 - \gamma)\theta_* \sqrt{p_{l}}}{ || \Psi({l},\mathbf{r}, \gamma \theta_*) ||_2 } \right) = \hat{\bm{\beta}}_l
\end{align*}
Let $v(\hat{\bm{\beta}}_l) \in \partial ||\hat{\bm{\beta}}_l ||_1$, then:
\begin{equation} \label{(3.2)}
    \left(1 - \frac{(1 - \gamma)\theta_* \sqrt{p_l} }{||\mathbf{r}_l + v(\hat{\bm{\beta}}_l) \gamma \theta_*||_2}\right) (\mathbf{r}_l + v(\hat{\bm{\beta}}_l) \gamma \theta_*) = \hat{\bm{\beta}}_l
\end{equation}
which implies
\begin{align*}
    ||\mathbf{r}_l + v(\hat{\bm{\beta}}_l) \gamma \theta_*||_2 = ||\hat{\bm{\beta}}_l||_2 / \left( 1 - \frac{(1 - \gamma)\theta_* \sqrt{p_l} }{||\mathbf{r}_l + v(\hat{\bm{\beta}}_l) \gamma \theta_*||_2} \right)
\end{align*}
then $||\mathbf{r}_l + v(\hat{\bm{\beta}}_l) \gamma \theta_*||_2 = ||\hat{\bm{\beta}}_l|| + (1 - \gamma) \theta_* \sqrt{p_l}$.
Plug in back to (\ref{(3.2)}), we have:
\begin{equation}
    \mathbf{r}_l - \hat{\bm{\beta}}_l = \theta_* \left[ \frac{(1 - \gamma)\sqrt{p_l}}{||\hat{\bm{\beta}}_l||_2} \hat{\bm{\beta}}_l + v(\hat{\bm{\beta}}_l) \gamma    \right]
\end{equation}
 is the subgradient of $ \theta_* \left((1 - \gamma) \sum_{l=1}^L \sqrt{p_l} ||\hat{\bm{\beta}}_l||_2 +  \gamma  ||\hat{\bm{\beta}}||_1\right)$. Hence, by the fixed point condition:
\begin{equation}
    \hat{\bm{\beta}} = \eta_\gamma(\mathbf{X}^\top \mathbf{z} + \hat{\bm{\beta}}, \theta_*)
\end{equation}
\begin{equation}
    \hat{\mathbf{z}} = \mathbf{y} - \mathbf{X}\hat{\bm{\beta}} + \mathbf{z}w.
\end{equation}
Set $\mathbf{r}_l = (\mathbf{X}^\top \mathbf{z})_l + \hat{\bm{\beta}}_l$, we have: $(\mathbf{X}^\top \mathbf{z})_l = \mathbf{r}_l - \hat{\bm{\beta}}_l$. 
Hence
\begin{equation}
  \left(    \mathbf{X}^\top (\mathbf{y} - \mathbf{X}\hat{\bm{\beta}}) \right)_l = (\mathbf{X}^\top \hat{\mathbf{z}})_l(1 - w_l) = (\mathbf{r}_l - \hat{\bm{\beta}}_l)   (1 -w_l ) 
\end{equation}
Comparing with the stationarity condition for the SGL cost function \eqref{cost}, we obtain.
\begin{align*}
    \lambda = \theta_* \left(1 - \frac{1}{\delta} \langle \eta_\gamma'(\mathbf{X}^\top \hat{\mathbf{z}} + \hat{\bm{\beta}}, \theta_*) \rangle   \right)
\end{align*}
\end{proof}

\begin{proof}[Proof of Proposition \ref{prop:bijective}]
The proof can be adapted from the LASSO case as in \cite{donoho2011noise}. When $\alpha \in \mathcal{A}$ where $\mathcal{A}$ is defined in Proposition \ref{prop:F}, $\tau_*^2$ is defined via 
\begin{align*}
    \tau_*^2 = \mathbf{F}_\gamma (\tau_*^2, \alpha \tau_*).
\end{align*}

Since $(\tau^2, \alpha) \rightarrow \mathbf{F}_\gamma (\tau^2, \alpha \tau)$ is continuously differentiable and $0 \leq \frac{d\mathbf{F}_\gamma}{d\tau^2}(\tau_*^2, \alpha \tau_*) < 1$ by Proposition \ref{prop:F}. Hence, we claim that $\tau_*^2(\alpha)$ is continuously differentiable in $\alpha$, which can be verified by applying the implicit function theorem to the function $(\tau^2, \alpha) \rightarrow \tau^2 - \mathbf{F}_\gamma (\tau^2, \gamma \alpha)$, whose derivative with respect to $\tau^2$ is bounded in $(0,1)$. 

Next, we denote $\alpha_{\min} = \min ( \mathcal{A})$ and $\alpha_{\max} = \max(\mathcal{A})$.

We show that $\tau_*^2(\alpha) \rightarrow \infty$ as $\alpha \downarrow \alpha_{\min}$. To see this, denote $\mathbf{F}_\infty' := \lim_{\tau^2 \rightarrow \infty} \frac{d\mathbf{F}_\gamma}{d\tau^2}(\tau^2, \alpha \tau) $. Then by concavity, 
\begin{align*}
    \tau_*^2 \geq \mathbf{F}_\gamma(0,0) + \mathbf{F}_\infty' \tau_*^2,
\end{align*}
or equivalently $\tau_*^2 \geq \mathbf{F}_\gamma (0,0)/ (1 - \mathbf{F}_\infty')$. As $\mathbf{F}_\gamma (0,0) \geq \sigma_w^2$ is bounded below and $\mathbf{F}_\infty' \uparrow 1$ as $\alpha \downarrow \alpha_{\min}$ (see the proof of Proposition \ref{prop:F}), the claim follows.

Now we define a function $q(\alpha, \tau^2)$ as:
\begin{align*}
    \alpha \tau \left[ 1 - \frac{1}{\delta} \mathbb{P} \left(|\Pi + \tau Z| > \gamma \alpha \tau  \right) \left(1 -  \frac{(1 - \gamma) \alpha }{\sqrt{G(z)} }            \right)             \right]
\end{align*}
such that where $G(z) = T(z) + T(-z + 2\gamma \alpha)$ with $T(z) $ being defined as in Proposition \ref{prop:F}, and $z := (\Pi + \gamma\alpha\tau )/\tau$. We note that the formula above corresponds to the asymptotic calibration $\lambda(\alpha) = g(\alpha, \tau_*^2(\alpha))$ and $G(z)$ is defined so as to denote $\mathbb{E}_Z[\eta_\text{soft}(\Pi+\tau Z;\gamma\alpha\tau)^2]/\tau^2$. We leave the detailed derivation as calculus exercises for interested readers.

It is not hard to see $q$ is continuously differentiable and therefore $\alpha \rightarrow \lambda(\alpha)$ is also continuously differentiable.

Define $z_*:=(\Pi+\gamma\alpha\tau_*)/\tau_*$ and
\begin{align*}
    l_\gamma(\alpha) := 1 - \frac{1}{\delta} \mathbb{P} (|\Pi + \tau_* Z| > \gamma \alpha \tau_*) \left(1 -  \frac{(1 - \gamma) \alpha }{\sqrt{G(z_*)}}\right).
\end{align*}
We have seen $\tau_* \rightarrow +\infty$ as $\alpha\rightarrow\alpha_{\min}$, and
\begin{align*}
    l_\gamma^* = \lim_{\alpha \rightarrow \alpha_{\min}^+} l_\gamma (\alpha)  = 1 - \frac{1}{\delta} \mathbb{P} (| Z| > \gamma \alpha_{\min} ) \left(1 -  \frac{(1 - \gamma) \alpha }{\sqrt{G(z_*)} }\right),
\end{align*}
also
\begin{align*}
    \lim_{\alpha \downarrow \alpha_{\min}}  \frac{1}{\sqrt{G(z)} }
    =\frac{1}{\sqrt{G(\gamma\alpha_{\min})} } = \frac{1}{\sqrt{2T(\gamma \alpha_{\min})}} 
\end{align*}
hence, 
\begin{align*}
    l_\gamma^* = 1 - \frac{2}{\delta} \Phi(-\gamma\alpha_{\min})\left(1 - \frac{(1-\gamma) \alpha_{\min}}{\sqrt{2T(\gamma \alpha_{\min})}}\right).
\end{align*}
We would show $l_\gamma^*<0$, which is equivalent to
\begin{align*}
\delta\leq 2\Phi(-\gamma\alpha_{\min})\left(1-\frac{(1-\gamma)\alpha_{\min}}{\sqrt{2T(\gamma \alpha_{\min}) }}\right).
\end{align*}

Recall from the definition of $\mathcal{A}(\gamma)$ in Proposition \ref{prop:F} that
\begin{align*}
\left(\sqrt{2T(\gamma \alpha_{\min})}-(1-\gamma)\alpha_{\min}\right)^2=\delta.
\end{align*}
It remains to show that, by suppressing the argument of $T$,
\begin{align*}
\sqrt{2T}\left(1-\frac{(1-\gamma)\alpha_{\min}}{\sqrt{2T}}\right)\left(\sqrt{2T}-(1-\gamma)\alpha_{\min}\right)
&\leq 2\Phi(-\gamma\alpha_{\min})\left(1-\frac{(1-\gamma)\alpha_{\min}}{\sqrt{2T}}\right)  
\\
\Longleftrightarrow
\sqrt{2T}\left(\sqrt{2T}-(1-\gamma)\alpha_{\min}\right)
&\leq 2\Phi(-\gamma\alpha_{\min})
\\
\Longleftrightarrow 2T(\gamma\alpha_{\min})&\leq 2\Phi(-\gamma\alpha_{\min})
\end{align*}
and the last inequality holds by observing that $z \Phi(-z) \leq \phi (z)$ for all $z > 0$.   Therefore, 
\begin{align*}
    \lim_{\alpha \downarrow \alpha_{\min}} \lambda(\alpha) =l_\gamma^*\alpha_{\min}     \lim_{\alpha \downarrow \alpha_{\min}} \tau(\alpha)= -\infty.
\end{align*}
Finally, consider $\alpha \uparrow \alpha_{\max}$ and we can see $\tau_*(\alpha)$ is bounded, leading to a finite upper bound of $\lambda$ which we denote as $\lambda_{\max}$.
\end{proof}

\begin{proof}[Proof of Corollary \ref{cor:monotone}]
Combining with Proposition \ref{prop:bijective}, we only need to show the uniqueness of $\alpha \in \mathcal{A}$ such that $\lambda(\alpha) = \lambda$. We prove this by contradiction and assume that there are $\hat\alpha_1\neq\hat\alpha_2$ such that $\lambda(\hat\alpha_1)=\lambda(\hat\alpha_2)=\lambda$.

We apply Theorem \ref{theorem1} to both choices $\alpha(\lambda) = \hat{\alpha}_1$ and $\alpha(\lambda) = \hat{\alpha}_2$ and use different test functions. We first use $\varphi(x,y) = (x - y)^2$ to obtain:
\begin{align*}
    \lim_{p \rightarrow \infty} \frac{1}{p} ||\hat{\bm \beta} - \bm \beta||^2 &= \mathbb{E} [\eta_\gamma(\Pi + \tau_*Z, \alpha \tau_*) - \Pi ]^2 = \delta(\tau_*^2 - \sigma^2).
\end{align*}
As $\hat{\bm\beta}$ only depends on $\lambda$ but not on $\hat\alpha_i$, we must have $\tau_*(\hat{\alpha}_1)=\tau_*(\hat{\alpha}_2)$.

Next, we apply Theorem \ref{theorem1} to $\varphi(x,y) = |x|$:
\begin{align*}
    \lim_{p \rightarrow \infty} \frac{1}{p} ||\hat{\bm \beta}||_1 = \mathbb{E} |\eta_\gamma(\Pi + \tau_* Z, \alpha \tau_*) |.
\end{align*}

For fixed $\tau_*$, it is obvious that the function $\mathbb{E}|\eta_\gamma(\Pi + \tau_* Z, \theta)|$ is strictly decreasing in $\theta$. Therefore, by looking at the thresholds, we conclude that $\hat{\alpha}_1 \tau_*(\hat{\alpha}_1) = \hat{\alpha}_2 \tau_*(\hat{\alpha}_2)$ and consequently $\hat{\alpha}_1 = \hat{\alpha}_2$. 
\end{proof}




\section{Extensions of SGL AMP} \label{Appendix G}
\label{app:extension}

As we have discussed in Section \ref{extensions}, the `i.i.d. Gaussian' assumption of the design matrix $\bm X$ may be relaxed. Here we explain the possibility of the relaxation in details.

\paragraph{AMP may work on i.i.d. non-Gaussian matrix}
We demonstrate two non-Gaussian design matrices for which AMP works. One example is a sub-Gaussian matrix following i.i.d. $\pm 1$ Bernoulli distribution (see Figure \ref{fig:iidBern,iidExp} (a)); the other example is a sub-exponential but not sub-Gaussian matrix following i.i.d. shifted exponential distribution (see Figure \ref{fig:iidBern,iidExp} (b)).

\begin{figure}[!htb]
\centering
\vspace{-0cm}
\hspace{-3cm}
\subfigure[]{\includegraphics[width=7cm]{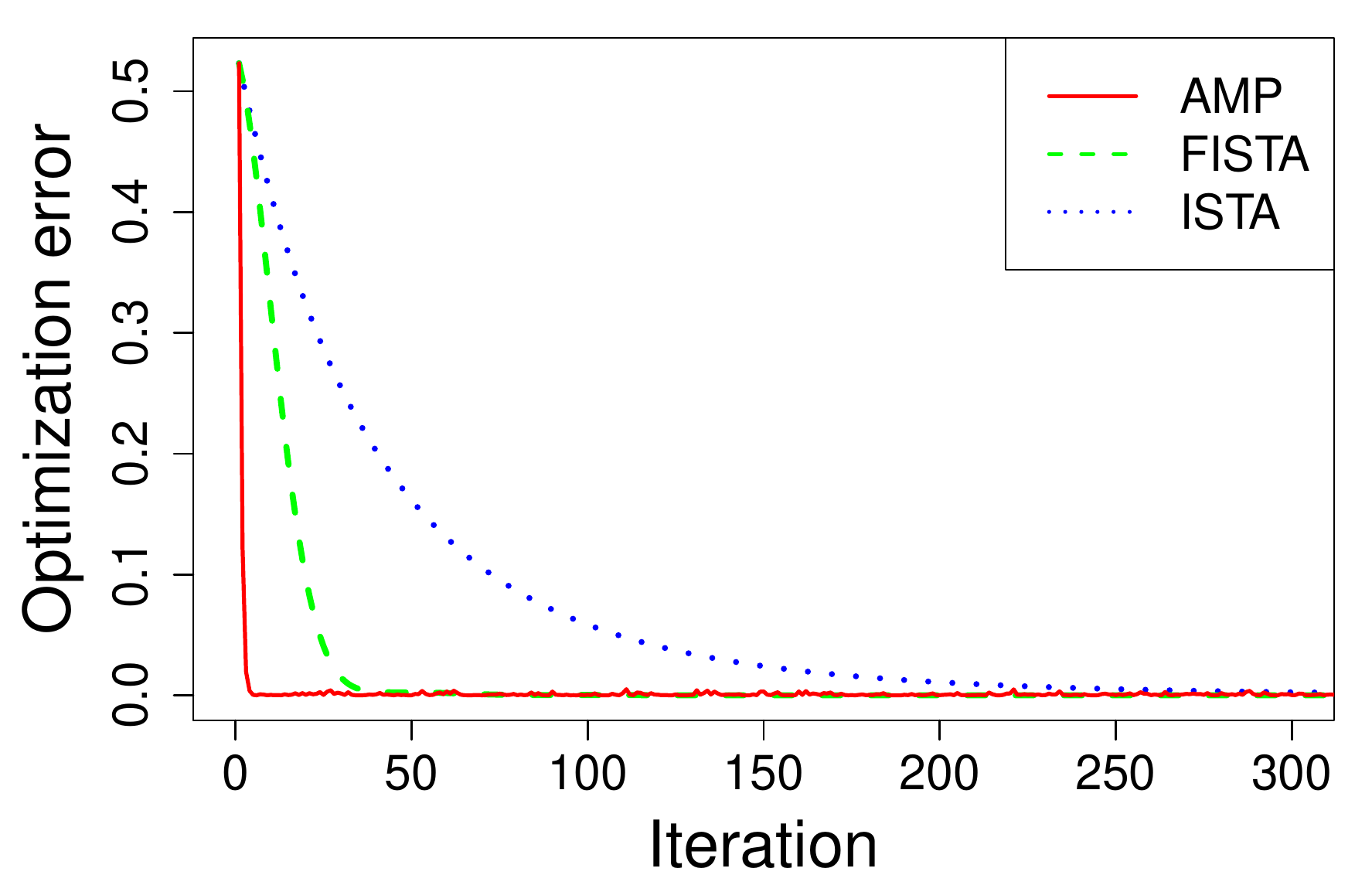}}
\subfigure[]{\includegraphics[width=7cm]{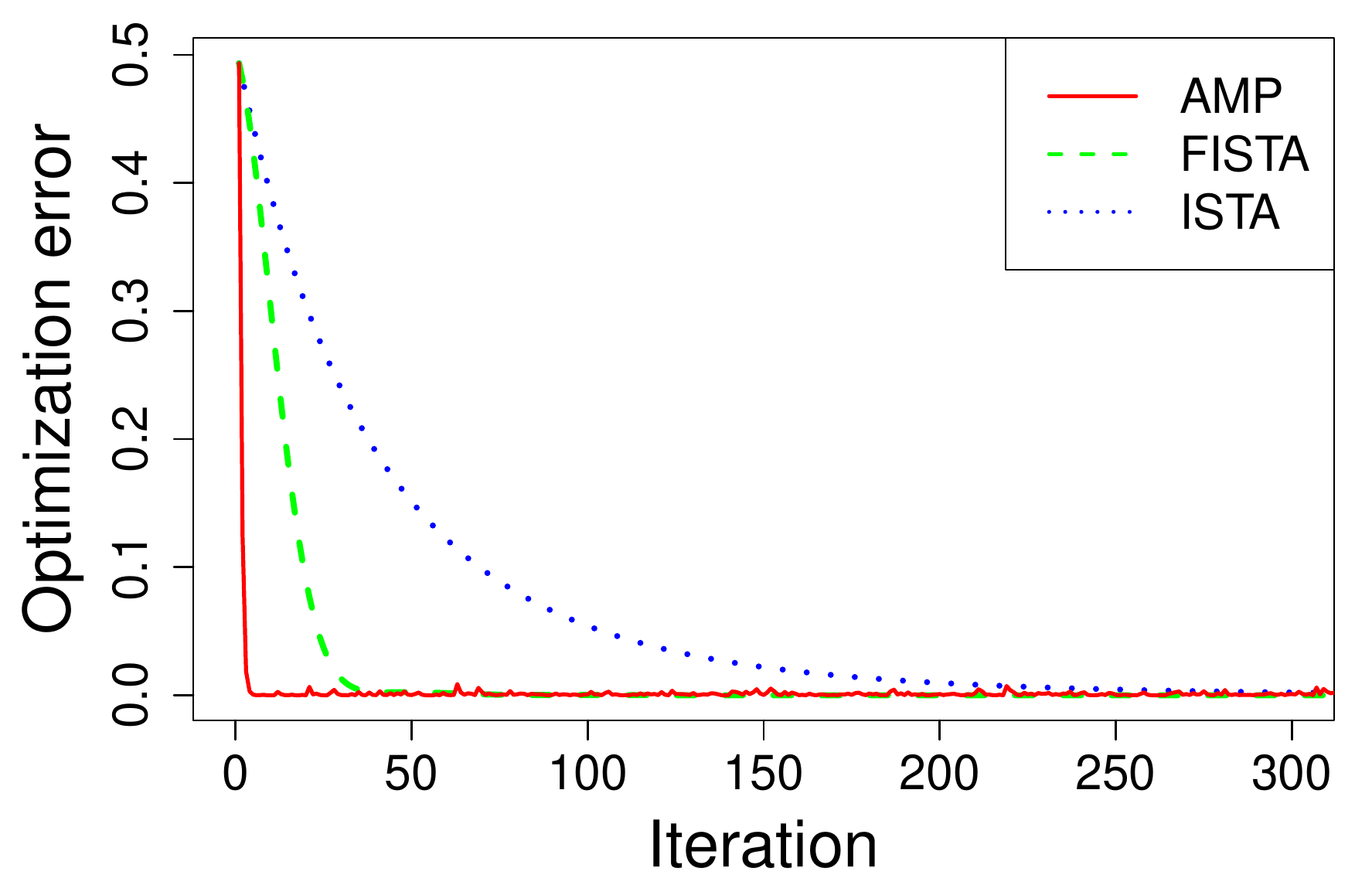}} \hspace{-3.2cm}
\vspace{-0.3cm}
\caption{(a) Same settings as in Figure \ref{fig:MSE1} except  the entries of $\mathbf{X}$ are i.i.d. scaled Bernoulli with mean $0$ and variance $\frac{1}{2000}$, i.e. $X_{ij}$ takes values in $\frac{1}{\sqrt{2000}}$ and -$\frac{1}{\sqrt{2000}}$ with equal probabilities. (b) Same settings as in Figure \ref{fig:MSE1} except the entries of $\mathbf{X}$ are i.i.d. shifted exponential with mean $0$ and variance $\frac{1}{2000}$: the probability density function is $\sqrt{n}e^{-\sqrt{n}x-1}$.}
    \label{fig:iidBern,iidExp}
\end{figure}

\paragraph{VAMP may work on non-i.i.d. matrix}
We present the SGL VAMP, modified from the LASSO VAMP by replacing the soft-thresholding with the SGL proximal operator. The empirical performance of VAMP is shown in Figure \ref{fig:vamp} . We note that certain techniques have be developed to help the algorithm converge. In this work, we adopt the damping technique \cite{rangana2019convergence,vila2015adaptive} with $\mathcal{D}=0.1$ for experiments.

\begin{algorithm}[H]
\centering
 \caption{SGL Vector AMP (VAMP)}
 \label{alg:vamp}
\begin{algorithmic}
\STATE{\textbf{Input}: Data matrix $\mathbf{X}$, response $\bm y$, group information $\bm g$, number of iterations $T$, damping constant $\mathcal{D}$.}
 \FOR{$t$ in $1:T$}
 \STATE{$\bm\beta^t = \left(\bm X^\top\bm X+\rho^t \mathcal{I}_p\right)^{-1}\left(\bm X^\top\bm y+\bm u^t\right)$
 \\
 $\sigma_\beta^t=\frac{1}{n}\textup{tr}\left(\bm X^\top\bm X+\rho^t\mathcal{I}_p\right)^{-1}$
 \\
 $\bm z^t = \eta_\gamma\left(\frac{\bm \beta^t-\sigma_\beta^t\bm u^t}{1-\sigma_\beta^t\rho^t}, \frac{\lambda\sigma_\beta^t}{1-\sigma_\beta^t\rho^t}\right)$
 \\
 $\sigma_z^t=\frac{\sigma_\beta^t}{1-\sigma_\beta^t}\left\langle\nabla\eta_\gamma\left(\frac{\bm \beta^t-\sigma_\beta^t\bm u^t}{1-\sigma_\beta^t\rho^t}, \frac{\lambda\sigma_\beta^t}{1-\sigma_\beta^t\rho^t}\right)\right\rangle$
 \\
 $\bm u^{t+1}=\bm u^t+(1-\mathcal{D})(\bm z^t/\sigma_z^t-\bm\beta^t/\sigma_\beta^t)$
 \\
 $\rho^{t+1}=\rho^t+(1-\mathcal{D})(1/\sigma_z^t-1/\sigma_\beta^t)$}
 \ENDFOR
\end{algorithmic}
\end{algorithm}